\documentclass[a4paper, 11pt]{amsart}
\usepackage{amsmath,amssymb,amscd}
\usepackage{float}
\usepackage{amsfonts}
\usepackage[english]{babel}
\usepackage[leqno]{amsmath}
\usepackage{amssymb,amsthm}
\usepackage{mathrsfs} 
\usepackage{amscd}
\usepackage{enumerate}
\usepackage{epsfig}
\usepackage{relsize}
\usepackage{layout}
\usepackage{fullpage}
\usepackage[usenames,dvipsnames]{xcolor}
\usepackage[backref=page]{hyperref}
\usepackage{tikz}
\hypersetup{
 colorlinks,
 citecolor=Green,
 linkcolor=Red,
 urlcolor=Blue}
\usepackage[matrix,arrow,tips,curve]{xy}
\input{xy}
\xyoption{all}
\definecolor{darkgreen}{rgb}{0.0, 0.7, 0.0}
\definecolor{purple}{rgb}{0.5, 0.0, 0.5}
\definecolor{red}{rgb}{0.8, 0.2, 0.0}

\newenvironment{sis}{\left\{\begin{aligned}}{\end{aligned}\right.}
\newtheorem{thm}{Theorem}[section]
\newtheorem{lemma}[thm]{Lemma}

\newtheorem{prop}[thm]{Proposition}
\newtheorem{cor}[thm]{Corollary}

\newtheorem{fact}[thm]{Fact}
\newtheorem{theoremalpha}{Theorem}

\numberwithin{equation}{section}
\theoremstyle{definition}
\newtheorem{defi}[thm]{Definition}

\newtheorem*{question*}{Question}

\theoremstyle{remark}
\newtheorem{remark}[thm]{Remark}
\newtheorem{example}[thm]{Example}
\newcommand{\Z}{\mathbb{Z}}

\newcommand{\R}{\mathbb{R}}

\newcommand{\Pic}{\operatorname{Pic}}

\newcommand{\ov}{\overline}

\def \Im{{\rm Im}}

\def \PP{\mathbb{P}}

\def \L{\mathcal L}

\def\O{\mathcal O}

\def\M0{\mathcal M^0}

\DeclareMathOperator{\codim}{{codim}}

\DeclareMathOperator{\Sym}{{Sym}}
\DeclareMathOperator{\contr}{{contr}}

\DeclareMathOperator{\AJ}{{AJ}}
\DeclareMathOperator{\AJt}{{^{\rm t} \hskip -.08cm \operatorname{AJ}}}

\DeclareMathOperator{\Kum}{{Kum}}

\newcommand{\bbN}{{\mathbb N}}

\newcommand{\bbQ}{{\mathbb Q}}

\newcommand{\g}{\mathfrak{g}}
\renewcommand{\l}{\mathfrak{l}}
\newcommand{\gon}{\operatorname{gon}}
\newcommand{\Eff}{\operatorname{Eff}}
\newcommand{\Pseff}{\operatorname{Pseff}}
\newcommand{\Nef}{\operatorname{Nef}}
\newcommand{\Amp}{\operatorname{Amp}}
\newcommand{\Efft}{{^{\rm t} \hskip -.05cm \operatorname{Eff}}}
\newcommand{\Psefft}{{^{\rm t} \hskip -.05cm \operatorname{Pseff}}}
\newcommand{\Neft}{{^{\rm t} \hskip -.05cm \operatorname{Nef}}}

\newcommand{\cone}{\operatorname{cone}}

\title{Effective cycles on the symmetric product of a curve, II: the Abel-Jacobi faces}

\author{Francesco Bastianelli, Alexis Kouvidakis, Angelo Felice Lopez and Filippo Viviani}

\thanks{Research partially supported by INdAM (GNSAGA) and by the MIUR national projects``Geometria delle variet\`a algebriche" PRIN 2010-2011 and ``Spazi di moduli e applicazioni" FIRB 2012.}

\address{\hskip -.43cm Dipartimento di Matematica, Universit\`a degli Studi di Bari Aldo Moro, Via Edoardo Orabona, 4, 70125 Bari (Italy). e-mail {\tt francesco.bastianelli@uniba.it}}

\address{\hskip -.43cm Department of Mathematics and Applied Mathematics, University of Crete, GR-70013 Heraklion, Greece. e-mail {\tt kouvid@math.uoc.gr}}

\address{\hskip -.43cm Dipartimento di Matematica e Fisica, Universit\`a di Roma Tre, Largo San Leonardo Murialdo 1, 00146 Roma, Italy. e-mail {\tt lopez@mat.uniroma3.it, filippo.viviani@gmail.com}}

\begin{document}


\begin{abstract}
In this paper, which is a sequel of \cite{BKLV}, we study the convex-geometric properties of the cone of pseudoeffective $n$-cycles in the symmetric product $C_d$ of a smooth curve $C$. We introduce and study the Abel-Jacobi faces, related to the contractibility properties of the Abel-Jacobi morphism and to classical Brill-Noether varieties. 
We investigate when Abel-Jacobi faces are non-trivial, and we prove that for $d$ sufficiently large (with respect to the genus of $C$) they form a maximal chain of perfect faces of the tautological pseudoeffective cone (which coincides with the pseudoeffective cone if $C$ is a very general 
curve). 
\end{abstract}

\maketitle

\section{Introduction}
\label{intro}

The study of the cone of ample or nef divisors, up to numerical equivalence, on a projective variety $X$ is a basic and classical tool in algebraic geometry, giving a lot of geometrical information about $X$. The dual cones in the space of $1$-cycles have also been deeply studied, giving rise, for example, to very important results in birational geometry and the minimal model program.

On the other hand, it is only in recent years that the study of higher dimensional (or codimensional) cycles has highlighted its role (see for example \cite{Pet, Voi, DELV, CC, DJV, Ful, fl2, fl1, fl3, Ott1, Ott2} to mention a few). 

One of the most striking features of higher codimensional cycles is that they behave in an unpredictable way, as there are examples of nef cycles with negative intersection \cite[Cor. 4.6]{DELV} or of nef cycles that are not pseudoeffective \cite[Cor. 2.2, Prop. 4.4]{DELV}, \cite[Thm. 1]{Ott3}. While one expects these phenomena not to be so special, there are so far only three examples of such varieties and 
it becomes therefore more interesting to investigate nef and pseudoeffective cycles for classical families of varieties, such as symmetric products of curves, as suggested in \cite[\S 6]{DELV}.

Let now $C$ be a smooth projective irreducible curve of genus $g$ and consider, for every $d \ge 2$, its $d$-fold
\color{black}symmetric product $C_d$,
which is the smooth projective variety parameterizing unordered $d$-tuples of points of C. \color{black}This is a very interesting smooth $d$-dimensional variety whose geometry has been deeply involved in the classical study of Brill-Noether theory \cite{GAC1} but also, in more recent years, in the investigation of cones of effective and nef divisors on it (an almost thorough recap of these results can be found in the introduction of \cite{BKLV}).

In this paper, which is a natural sequel of \cite{BKLV}, we study cones of pseudoeffective and nef cycles on $C_d$. To state our results we need to set up some notation and recall some well-known facts.

For $0 \le n \le d$, let $N_n(C_d)$ be the vector space of real $n$-cycles up to numerical equivalence. 
Inside this finite dimensional real vector space one can define several interesting cones, namely $\Eff_n(C_d)$, the \emph{cone of effective $n$-cycles}, its closure $\Pseff_n(C_d)$, the \emph{cone of pseudoeffective $n$-cycles} and $\Nef_n(C_d)$, the \emph{cone of nef $n$-cycles}, that is cycles $\alpha \in N_n(C_d)$ such that $\alpha \cdot \beta \ge 0$ for every $\beta \in \Eff_{d-n}(C_d)$. 

One crucial feature of $C_d$, that will play a very important role in this paper, is that it comes naturally with a well-known map, the Abel-Jacobi morphism
$$\alpha_d\colon C_d \to \Pic^d(C)$$ 
given by sending an effective divisor $D\in C_d$ to its associated line bundle $\O_C(D)\in \Pic^d(C)$.

On $C_d$, there are two important divisor classes, up to numerical equivalence:
$$x = [\{D \in C_d : D = p_0+D', D' \in C_{d-1}\}] \quad \text{and} \quad \theta=\alpha_d^*([\Theta]),$$
where $p_0$ is a fixed point of $C$ and $\Theta$ is any theta divisor on $\Pic^d(C)$. 
It is well-known that $x$ is ample while $\theta$ is clearly nef being the pull-back of an ample class. 

These two classes generate a graded subring $R^*(C_d)=\oplus_m R^m(C_d)$ of $N^*(C_d)$, which is called the \emph{tautological} ring of cycles, whose structure is well-understood (and independent of the given curve $C$), and which coincides with the full ring $N^*(C_d)$ if $C$ is a very general
curve, see \cite[Fact 2.6]{BKLV}. We will also consider the natural cones $\Efft_n(C_d)$ generated by \emph{tautological effective cycles of dimension $n$}, its closure $\Psefft_n(C_d)$, called the \emph{cone of tautological pseudoeffective cycles of dimension $n$}, and $\Neft_n(C_d)$, the \emph{cone of tautological nef cycles of 
dimension \color{black}$n$}. 

The main result of \cite{BKLV}, generalizing the case of divisors and curves, was to prove that the cone generated by $n$-dimensional diagonals is a rational polyhedral perfect face of $\Pseff_n(C_d)$ and that $\Pseff_n(C_d)$ is locally finitely generated at every non-zero element of that cone \cite[Thm. B]{BKLV}.
On one side this gives a very nice face of $\Pseff_n(C_d)$, but, on the other side, it opens the way to look for other faces.

In the case of divisors and curves, the situation is well-understood if $d$ is sufficiently large:
\begin{itemize}
\item the other extremal ray of $\Psefft^1(C_d)$ is generated by $\theta$ if and only if $d\geq g+1$. Indeed, 
$\theta$ is always pseudoeffective (being nef) and it is not in the interior of the pseudoeffective cone, i.e. it is not big, if and only if $\alpha_d$ is not birational into its image, which happens exactly when $d \geq g+1$.
\item The other extremal ray of $\Psefft_1(C_d)$ is generated by the ray dual to $\R_{\geq 0}\cdot \theta$ (or equivalently, $\theta$ generates an extremal ray of $\Neft^1(C_d)$) if and only if $d\geq \gon(C)$ where $\gon(C)$ is the gonality of $C$. Indeed, $\theta$ is always nef and it is not in the interior of the nef cone, i.e. it is not ample, if and only if $\alpha_d$ is not a finite morphism, which happens exactly when $d \geq \gon(C)$.
\end{itemize} 

The aim of this paper is to generalize the extremality properties of $\theta$ for $\Psefft^1(C_d)$ and $\Neft^1(C_d)$ to the case of cycles of intermediate codimension. 
As a matter of fact the other faces that we will find will all come from the contractibility properties of the Abel-Jacobi morphism, as we now explain.

Given any morphism $\pi : X \to Y$ between irreducible projective varieties, in \cite[\S 4.2]{fl2} was introduced the \emph{contractibility index} 
$\contr_{\pi}(\alpha)$ \color{black}of a class $\alpha\in \Pseff_k(X)$, for $0 \le k \le \dim X$, as the largest 
integer \color{black}$0 \le c \le k+1$ such that $\alpha\cdot \pi^*(h^{k+1-c})=0$, where $h$ is an ample class on $Y$.
This gives immediately rise, for every $r \ge 0$, to the \emph{contractibility faces} of $\Pseff_k(X)$:
$$F_k^{\geq r}(\pi) = \cone(\{\alpha \in \Pseff_k(X) : \contr_{\pi}(\alpha)\ge r \}).$$ 
The main question about them is to identify for which $r$ such that $1 + \max\{0, k-\dim \pi(X)\} \le r \le k$ we have that $F_k^{\geq r}(\pi)$ is non-trivial and, in that case, to compute its dimension and convex-geometrical properties.

With this in mind, for any $1+\max\{0,g-n\}\leq r \leq  n$, we define the \emph{Abel-Jacobi faces} 
$$\AJ_n^r(C_d) = F_n^{\ge r}(\alpha_d) \subseteq \Pseff_n(C_d).$$

From the general properties of the contractibility faces and the classical properties of the Brill-Noether varieties (which are reviewed in \S \ref{SS:BNvar}) $C_d^r:=\{D\in C_d: \: \dim |D|\geq r\}$, we prove in Proposition \ref{P:nonzero}  that:
\begin{itemize}
\item $\AJ_n^r(C_d)$ is non-trivial if $n \le \dim C_d^r$;
\item $\AJ_{\dim C_d^r}^r(C_d)$ is  the conic hull  of the irreducible components of $C_d^r$ of maximal dimension.
\end{itemize}
Intersecting with the tautological ring, we can also define the  \emph{tautological Abel-Jacobi faces} 
$$\AJt_n^r(C_d):=\AJ_n(C_d)\cap R_n(C_d)\subseteq \Psefft_n(C_d).$$






The following theorem (which combines Corollary \ref{C:nontriv}, Theorem  \ref{AJtaut} and Proposition \ref{P:facet}) specifies some numerical ranges where we can find 
non-trivial Abel-Jacobi faces, Abel-Jacobi facets and Abel-Jacobi extremal rays coming from well-known facts of Brill-Noether theory (on Brill-Noether general curves).


\begin{theoremalpha}\label{T:A}
Let $C$ be a curve of genus $g$. 
\begin{enumerate}[(1)]
\item \label{T:A1} If $1 \le n \le d-1$ and $\displaystyle d\geq \frac{n+g+1}{2}$ then  there 
exist non-trivial Abel-Jacobi faces of $\Pseff_n(C_d)$. The same is true for $\Psefft_n(C_d)$ if either $d \ge g+1$ or 
$\displaystyle d\geq \frac{n+g+1}{2}$ and $C$ is a Brill-Noether general curve, i.e. it satisfies the condition in Fact \ref{BNclass}\eqref{BNclass3}.
\item \label{T:A3} We have tautological Abel-Jacobi facets in the following ranges:
\begin{enumerate}[(i)]
\item If $g\leq n$ then $\AJt^{n+1-g}_n(C_d)$ is a facet of $\Psefft_n(C_d)$.
\item  If $n\leq g$ then $\AJt^1_n(C_d)$ is a facet of $\Psefft_n(C_d)$ under one of the following assumptions:
\begin{enumerate}[(a)]
\item $C$ admits a $\g_d^n$, i.e. a linear series of dimension $n$ and degree $d$   (which is always satisfied if $g\leq d-n$);

\item $n=g-1$;
\item \label{T:Ab3} $g\leq d$ and $C$ is very general over an uncountable base field $k$.
\end{enumerate}
\end{enumerate} 


\item \label{T:A2} Assume that 
$C$ is Brill-Noether general. 
Let $0\leq d\leq 2g-2$ and let $r$ be an integer such that $\max\{1,d-g+1\}\leq  r$ and $\rho:=\rho(g,r,d)=g-(r+1)(g-d+r)\geq 0$.
Then 
$$\AJ^r_{r+\rho}(C_d)=\AJt^r_{r+\rho}(C_d)=\cone([C_d^r]).$$  
 In particular, $[C_d^r]$ generates an extremal ray (called the \emph{BN(=Brill-Noether) ray}) of $\Pseff_{r+\rho}(C_d)$ and of $\Psefft_{r+\rho}(C_d)$. 
\end{enumerate}
\end{theoremalpha}


Note that, for a Brill-Noether general curve $C$, if $r=1$ and $\frac{g+2}{2}\leq d \leq g$ then $[C_d^1]$ generates an extremal ray of $\Psefft_{2d-g-1}(C_d)$, and this achieves the lower bound on $d$ in Theorem \ref{T:A}\eqref{T:A1}. On the other hand, we expect that  the lower bound $d\geq \frac{n+g+1}{2}$ is sharp for Brill-Noether general curves (see the discussion after Corollary \ref{C:nontriv}), while for special curves the lower bound is far from being sharp 
 (see Theorem \ref{T:C} for  hyperelliptic curves).

The tautological Abel-Jacobi faces are related to an exhaustive  decreasing multiplicative filtration of the tautological ring $R^*(C_d)$, which we call 
the \emph{$\theta$-filtration} of $R^*(C_d)$ (see \S \ref{SS:theta}) and which is defined by setting $\theta^{\geq i, m}$  to be the smallest linear subspace of $R^m(C_d)=R_{d-m}(C_d)$ containing the monomials $\{\theta^i x^{m-i}, \theta^{i+1}x^{m-i-1}, \ldots, \theta^{m} \}$, for any  $0\leq m \leq d$ and any $0\leq i \leq g+1$ (with the obvious convention that $\theta^{\geq i, m}=\{0\}$ if $i>m$). In Proposition \ref{P:theta}, we compute the dimension of  $\theta^{\geq i, m}$ and we investigate its orthogonal subspace $(\theta^{\geq i, m})^\perp:=\{\alpha\in R_m(C_d): \alpha\cdot \beta=0 \: \text{ for any } \beta\in \theta^{\geq i,m}\}\subseteq R_m(C_d)$. 
The link between tautological Abel-Jacobi faces and the $\theta$-filtration is explained in Proposition \ref{P:AJtheta}, where we prove that   
$$\AJt_n^r(C_d)= (\theta^{\geq n+1-r,n})^\perp \cap \Psefft_n(C_d)\subset (\theta^{\geq n+1-r,n})^\perp,$$
and that if  $\AJt_n^r(C_d)$ is a full-dimensional cone in the linear subspace $(\theta^{\geq n+1-r,n})^\perp$ then  $\AJt_n^r(C_d)$ is a perfect face of $\Psefft_n(C_d)$ whose (perfect) dual face is  $\theta^{\geq n+1-r,n} \cap \Neft^n(C_d)$ (faces of this kind are called \emph{nef $\theta$-faces}). 



Using the relation  with the $\theta$-filtration, we are able to show that in many ranges of $d$ and $n$ the non-trivial tautological Abel-Jacobi faces of $\Psefft_n(C_d)$ form  a \emph{maximal chain of perfect non-trivial faces}, i.e. a chain of perfect non-trivial faces of $\Psefft_n(C_d)$ whose dimensions start from one and increase by one at each step until getting to the dimension of $\Psefft_n(C_d)$ minus one. In the following theorem (which combines Theorems \ref{thetaPseff}, \ref{thetaNef}, \ref{T:BNfaces}), we summarize the cases where this happens for an arbitrary curve. 


\begin{theoremalpha}\label{T:B}
Let $n, d$ be integers such that $0 \le n \le d$.
\begin{enumerate}
\item \label{T:B1} Assume that $g \le \max\{n,d-n\}$ (which is always satisfied if $d>2g-2$).
Then the Abel-Jacobi face $\AJt_n^r(C_d)$ is equal to $\theta^{\geq g-n+r, d-n} \cap \Psefft_n(C_d)$ (and we call it \emph{pseff $\theta$-face}) and it is non-trivial  if and only if $1 + \max\{0,n-g\}\leq r\leq \min\{n,d-g\}$, in which case it is a perfect face of dimension $\min\{n,d-g\}-r+1$.

Hence, we get the following dual maximal chains of perfect non-trivial faces of $\Psefft_n(C_d)$ and of $\Neft^n(C_d)$:

$$
\begin{sis}
& \theta^{\geq {\min\{g,d-n\},d-n}}\cap \Psefft_n(C_d)=\cone(\theta^{\min\{g,d-n\}}x^{\max\{d-n-g,0\}})\subset  \dots \subset\\
& \qquad\qquad\qquad\qquad\qquad\qquad\qquad\quad\subset\dots\subset \theta^{\geq g+1-\min\{g,n\},d-n}\cap \Psefft_n(C_d)\subset \Psefft_n(C_d),\\
&\theta^{\geq {\min\{g,n\},n}}\cap \Neft^n(C_d)=\cone(\theta^{\min\{g,n\}}x^{\max\{n-g,0\}})\subset  \ldots \subset \\
& \qquad\qquad\qquad\qquad\qquad\qquad\qquad\qquad\subset\dots\subset \theta^{\geq g+1-\min\{g,d-n\},n}\cap \Neft^n(C_d)\subset \Neft^n(C_d).
\end{sis}
$$

\item \label{T:B2} Assume that   $C_d^n\neq \emptyset$ (which is always satisfied if $\displaystyle d\geq \frac{ng}{n+1}+n$).
Then  $\AJt_n^r(C_d)$ is a non-trivial face if and only if $1+\max\{0,n-g\}\leq r \leq n$, in which case $\AJt_n^r(C_d)$ is a perfect face of dimension $n+1-r$ (which we call  \emph{subordinate face}).

Hence, we get the following dual maximal chains of perfect non-trivial faces of $\Psefft_n(C_d)$ and of $\Neft^n(C_d)$:
$$
\begin{sis}
& \AJt_n^n(C_d)=\cone([\Gamma_d(\l)])\subset  \ldots \subset \AJt_n^{n+1-\min\{n,g\}}(C_d) \subset \Psefft_n(C_d),\\
& \theta^{\ge \min\{n,g\},n}\cap \Neft^n(C_d)=\cone(\theta^{\min\{g,n\}}x^{\max\{n-g,0\}})  \subset \ldots \subset \theta^{\ge 1,n} \cap \Neft^n(C_d) \subset \Neft^n(C_d),
\end{sis}
$$
where $\l$ is any $\g_d^n$ on $C$ and $\Gamma_d(\l)$ is the subordinate variety
$$\Gamma_d(\l):=\{D\in C_d: \:  D\leq E \text{ for some }E\in \l\}\subset C_d.$$

\item \label{T:B3} Assume that $g\leq d\leq 2g-2$.
Then  $\AJt_{g-1}^r(C_d)$ is a non-trivial face if and only if $1\leq r \leq d-g+1$, in which case $\AJt_{g-1}^r(C_d)$ is a perfect face of dimension $d-g+2-r$ (which we call \emph{BN(=Brill-Noether) face in dimension $g-1$}).

Hence, we get the following dual maximal chains of perfect non-trivial faces of $\Psefft_{g-1}(C_d)$ and of $\Neft^{g-1}(C_d)$:
$$
\begin{sis}
& \AJt_{g-1}^{d-g+1}(C_d)=\cone([C_d^{d-g+1}])\subset  \ldots \subset \AJ_{g-1}^{1}(C_d) \subset \Psefft_{g-1}(C_d),\\
& \theta^{\ge g-1,g-1}\cap \Neft^{g-1}(C_d)=\cone(\theta^{g-1}) \subset \ldots \subset \theta^{\ge 2g-1-d,g-1} \cap \Neft^{g-1}(C_d) \subset \Neft^{g-1}(C_d).
\end{sis}
$$
\end{enumerate}
\end{theoremalpha}

There are some overlaps between the different cases of the above Theorem \ref{T:B}, see Remarks \ref{sharpNef} and \ref{R:compaAJ}. 
In Figure \ref{uno} we present the existence range of the various 
tautological Abel-Jacobi faces in Theorems \ref{T:A} and \ref{T:B} 
for Brill-Noether general curves\color{black}.



Note that we recover from Theorem \ref{T:B} the previously mentioned results of the extremality of $\theta$ in $\Psefft^1(C_d)$ and $\Neft^1(C_d)$: part \eqref{T:B1} gives that $\theta$ is extremal in $\Psefft^1(C_d)$ if $d\geq g+1$ and part \eqref{T:B2} gives that $\theta$ is extremal in $\Neft^1(C_d)$ if $C_d^1\neq \emptyset$, i.e. if $d\geq \gon(C)$. Moreover, both  results are sharp. Furthermore, part \eqref{T:B3} of Theorem \ref{T:B} gives that $[C_{g}^1]=\theta-x\in N^1(C_g)$ generates an extremal ray of $\Pseff^1(C_{g})$.  This extends the result of the second author (see \cite[Rmk 1 after Thm.\ 5]{Kou}) from very general complex curves to arbitrary curves 
over an algebraically closed field. 

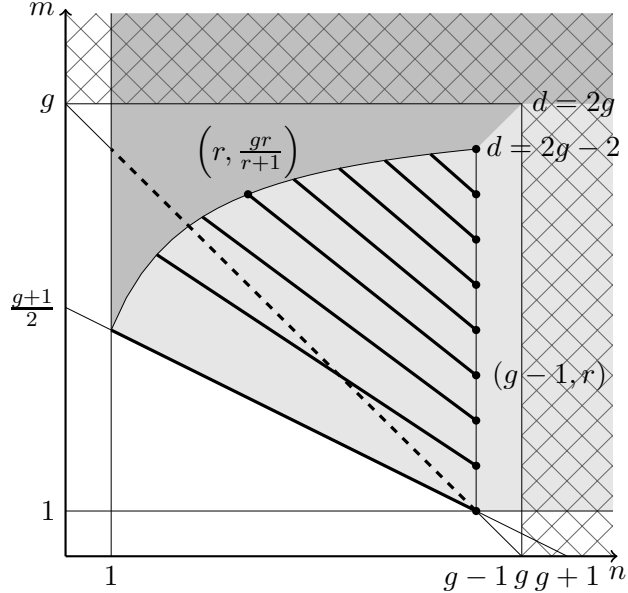
\begin{figure}[H]
\begin{tikzpicture}[scale=.6]                                    

\begin{scope}
\clip (1,5) -- (9,1) -- (12,1) -- (12,12) -- (1,12) -- cycle;
\fill[black!10!white] (0,0) rectangle (12,12);
\end{scope}

\begin{scope}
\clip (1,5) -- plot[domain=1:9] (\x,{(10*(\x))/((\x)+1)}) -- (10,10) -- (12,10) -- (12,12) -- (1,12) -- cycle;
\fill[black!25!white] (0,0) rectangle (12,12);
\end{scope}

\begin{scope}
\clip (0,10) -- (10,10) -- (10,0) -- (12,0) -- (12,12) -- (0,12) -- cycle;
\draw[very thin, black!60!white, step=.471,rotate=45] (0,-12) grid (20,12);
\end{scope}

\draw[color=black, thick, ->] (0,0) -- (12.1,0) node[below] {$n$};            
\draw[color=black, thick, ->] (0,0) -- (0,12.1) node[left] {$m$};                
\draw[very thin, domain=1:9] plot (\x,{(10*(\x))/((\x)+1)});       
\draw[very thin] (0,11/2) node[left] {$\frac{g+1}{2}$} -- (11,0) node[below] {$g+1$};
\draw[very thin] (0,1) node[left] {1} -- (12,1);     
\draw[very thin] (1,0) node[below] {1} -- (1,12);     
\draw[very thin] (9,1) -- (9,9);                                                  
\draw[very thin] (10,0) -- (9,1);   
\draw[very thin] (0,10) -- (1,9);
\draw[very thick, dashed] (9,1) -- (1,9);   
\draw[very thin] (10,0) -- (10,10);   
\draw[very thin] (0,10) -- (10,10);   
\draw[very thick] (1,10/2) -- (9,1) node [fill, color=black, shape=circle, scale=0.3]{} ;
\draw[very thick] (2,20/3) -- (9,2) node [fill, color=black, shape=circle, scale=0.3]{} ;
\draw[very thick] (3,30/4) -- (9,3) node [fill, color=black, shape=circle, scale=0.3]{} ;
\draw[very thick] (4,40/5) node [fill, color=black, shape=circle, scale=0.3, label=above:${\left(r,\frac{gr}{r+1}\right)}$]{}  -- (9,4) node [fill, color=black, shape=circle, scale=0.3, label=right:${\left(g-1,r\right)}$]{} ;
\draw[very thick] (5,50/6) -- (9,5) node [fill, color=black, shape=circle, scale=0.3]{} ;
\draw[very thick] (6,60/7) -- (9,6) node [fill, color=black, shape=circle, scale=0.3]{} ;
\draw[very thick] (7,70/8) -- (9,7) node [fill, color=black, shape=circle, scale=0.3]{} ;
\draw[very thick] (8,80/9) -- (9,8) node [fill, color=black, shape=circle, scale=0.3]{} ;
\draw (9,9) node [fill, color=black, shape=circle, scale=0.3]{} ;

\draw (9,0) node[below] {$g-1$};
\draw (10,-0.08) node[below] {$g$};
\draw (0,10) node[left] {$g$};
\draw (10,10) node[right] {$d=2g$};
\draw (9,9) node[right] {$d=2g-2$};
\end{tikzpicture}
\caption{The picture describes the existence of tautological Abel-Jacobi faces when $C$ is a Brill-Noether general curve.
We set $m=d-n$. 
The colored area is defined by the inequality $d\geq \frac{n+g+1}{2}$ in Theorem \ref{T:A}\eqref{T:A1}, so it describes the locus where we can assure the existence of non-trivial Abel-Jacobi faces.
In particular, the dark gray area is given by Theorem \ref{T:B}\eqref{T:B2} and represents the locus where subordinate faces do exist.
On the other hand, BN rays exist on the integral points of the thick lines---each having equation $(r+1)m+rn=r^2+rg$ for some $1\leq r\leq g$---by Theorem \ref{T:A}\eqref{T:A2}, and the dots on the line $n=g-1$ indicate where BN faces in dimension $g-1$ exist, according to Theorem  \ref{T:B}\eqref{T:B3}.
The area covered by the grid is described by the condition $g\leq \max\{n,m\}$ in Theorem  \ref{T:B}\eqref{T:B1} which governs the existence of $\theta$-faces. 
Finally,  Theorem \ref{T:A}\eqref{T:A3} guarantees the existence of Abel-Jacobi facets in the union of the area covered by the grid with the dark gray area and the line $n=g-1$.
If in addiction $C$ is assumed to be very general, then Abel-Jacobi facets do exist also in the area on the right and above the dashed line, which is the locus satisfying the condition $n+m\geq g$.}
\label{uno}
\end{figure}

According to the discussion at the end of \S\ref{SS:theta}, there is one more case (apart from the three cases of Theorem \ref{T:B}) where the non-trivial tautological Abel-Jacobi faces form a maximal chain of non-trivial perfect faces, namely the case of an hyperelliptic curve where we found such maximal chains in every pseudoeffective cone! This is made precise in the following theorem (which summarizes the more precise Theorem \ref{T:hyper}) where we restrict to the case $n, d-n<g$ since in the remaining case $g\leq \max\{n,d-n\}$ everything follows from part \eqref{T:B1} of Theorem \ref{T:B}.  

\begin{theoremalpha}\label{T:C}
Let $C$ be an hyperelliptic curve of genus $g\geq 2$ and fix integers $n,d$ such that $0 \le n, d-n <g$ (which implies that $0\leq d \leq 2g-2$). 
\begin{enumerate}[(i)]
\item \label{T:C1} Assume that $d\geq 2n$. 

Then, for any $1\leq r\leq n$,  $\AJt_n^r(C_d)$ is a non-trivial face, in which case $\AJt_n^r(C_d)$  is a perfect face of dimension $n+1-r$.
Hence, we get the following dual maximal chains of perfect non-trivial faces of $\Psefft_n(C_d)$ and of $\Neft^n(C_d)$:
$$
\begin{sis}
& \AJt_n^n(C_d)=\cone([\Gamma_d(\l)])\subset  \ldots \subset \AJt_n^{1}(C_d) \subset \Psefft_n(C_d),\\
& \theta^{\ge n,n}\cap \Neft^n(C_d)=\cone(\theta^{n})  \subset \ldots \subset \theta^{\ge 1,n} \cap \Neft^n(C_d) \subset \Neft^n(C_d),
\end{sis}
$$
where $\Gamma_d(\l)$ is the subordinate variety with respect to any linear system $\l$ of degree $d$ and dimension $n$. 

\item \label{T:C2} Assume that $d\leq 2n$.

Then $\AJt_{n}^r(C_d)$ is a non-trivial face if and only if $1\leq r \leq d-n$, in which case $\AJt_{n}^r(C_d)$ is a perfect face of dimension $d-n+1-r$ (which we call \emph{hyperelliptic BN(=Brill-Noether) 
face\color{black}}).
Hence, we get the following dual maximal chains of perfect non-trivial faces of $\Psefft_n(C_d)$ and of $\Neft^n(C_d)$:
$$
\begin{sis}
& \AJt_n^{d-n}(C_d)=\cone([C_d^{d-n}])\subset  \ldots \subset \AJt_n^{1}(C_d) \subset \Psefft_n(C_d),\\
& \theta^{\ge n,n}\cap \Neft^n(C_d)=\cone(\theta^{n})  \subset \ldots \subset \theta^{\ge 2n-d+1,n} \cap \Neft^n(C_d) \subset \Neft^n(C_d).
\end{sis}
$$
\end{enumerate}
\end{theoremalpha}

Note that the tautological Abel-Jacobi faces of part \eqref{T:C1} are exactly the subordinate faces of part \eqref{T:B2} of Theorem \ref{T:B}, using that for an hyperelliptic curve $C$ we have that $C_d^n\neq \emptyset $  precisely when $d\geq 2n$. 

From Theorem \ref{T:C} we recover the previously known results for $\Psefft^1(C_d)$ and $\Neft^1(C_d)$ for $C$ hyperelliptic: part 
\eqref{T:C1} gives that $\theta$ is extremal in $\Neft^1(C_d)$ for any $d\geq 2=\gon(C)$, part \eqref{T:C2} gives that $C_d^{d-1}$, whose class is a positive multiple of $\theta-(d-g+1)x$ by Proposition \ref{P:BNhyper}, is extremal in $\Psefft^1(C_d)$, which was proved in  \cite[Prop. H]{Mus1}.


On the other hand we should point out that there are some explicit extremal rays of $\Pseff_1(C_{g/2})$ in \cite[Thm. 4.1]{Pac} and of 
$\Pseff^1(C_d), d = g-1, g-2$ in \cite[Thm. A(ii)]{Mus1} and in \cite[Thm. I]{Mus2} which are not Abel-Jacobi faces.

\vspace{0.1cm}

The present work leaves open some natural questions.

\begin{question*}
Assume that we are in one of the cases of Theorem \ref{T:B}.
\begin{enumerate}
\item What is the structure of the Abel-Jacobi faces of dimension greater than one? Are they rational polyhedral cones and, if yes, what are their extremal rays? 
\item Is $\Psefft_n(C_d)$ the smallest cone containing the diagonal cone and the Abel-Jacobi faces? 
\end{enumerate}
\end{question*}



\section{Preliminaries}

\subsection{Notations and conventions}

Throughout, we work over an algebraically closed field of arbitrary characteristic.



For any natural number $n\in \bbN$ and any real number $r\in \R$, we set 
$$\binom{r}{n}=
\begin{cases}
\frac{r(r-1)\ldots (r-n+1)}{n!} & \text{ if }n>0, \\
1 & \text{ if } n=0.
\end{cases}
$$

We recall from the Appendix of \cite{BKLV} a definition and a remark, that will be crucial in this paper.


\begin{defi}
Let $V$ be a finite dimensional real vector space. A (convex) \textbf{cone} $K$ inside $V$ is a non-empty subset $K$ of $V$ such that if $x,y\in K$  and $\alpha, \beta\in \R^{>0}$ then $\alpha x+\beta y\in K$. A face $F$ of $K$ is {\bf perfect} if either $F = K$ or it has codimension $c \ge 1$ in $V$ and there exist linear hyperplanes $H_i =\{l_i=0\}_{1\leq i \leq c}$, such that 
\begin{equation}\label{E:perf-face}
\begin{sis}
& K\subseteq H_i^+=\{l_i \geq 0\}\quad \text{for any }\: 1\leq i \leq c, \\
& \langle F \rangle=\cap_{i=1}^c H_i.
\end{sis}
\end{equation}
A cone $K\subset V$ is \textbf{salient} if it does not contain lines through the origin, and it is \textbf{full} if $\langle K\rangle=V$.\color{black}
\end{defi}

\begin{remark}\label{perf}
Let $K\subset V$ be a salient full closed cone 
and let $K^{\vee}:=\{l\in V^{\vee}|\,l(x)\geq 0\, \forall x\in K\}$ be its dual cone\color{black}. If $L\subseteq V$ is a subspace such that $K\cap L$ is a full cone in $L$ and 
$K^{\vee}\cap L^{\perp}$ is a full cone in $L^{\perp}$, then $F:=K\cap L$ is a perfect face of $K$ with dual face being perfect and equal to $K^{\vee}\cap L^{\perp}$.   
\end{remark}

\subsection{Symmetric product}

Let $C$ be a smooth projective irreducible curve of genus $g\geq 1$. For any integer $d\geq 1$, we denote by $C^d$ the 
$d$-fold ordinary \color{black} product of $C$ and by $C_d$ the 
$d$-fold \color{black} symmetric product of $C$. 

The symmetric product $C_d$ is related to the Jacobian of $C$ by the 
Abel-Jacobi morphism 
\begin{equation*}
\begin{aligned}
\alpha_d:C_d & \longrightarrow \Pic^d(C)\\
D& \mapsto \O_C(D).  
\end{aligned}
\end{equation*}
The fiber of $\alpha_d$ over $L\in \Pic^d(C)$ is the complete linear system $|L|$. 

Fixing a base point $p_0\in C$, there is an inclusion $i=i_{p_0}: C_{d-1} \hookrightarrow C_d$, obtained by sending $D$ into $D+p_0$. We will denote by $X_{p_0}$ the image of $i_{p_0}$. 
The inclusion $i_{p_0}$ is compatible with the Abel-Jacobi morphisms in the sense that  $\alpha_d\circ i_{p_0}= t_{p_0}\circ \alpha_{d-1}$, where $t_{p_0}:\Pic^{d-1}(C)\to \Pic^d(C)$ is the translation by $p_0$ which 
sends $L$ into $L(p_0)$.  

\subsection{Tautological ring}\label{SS:tautring}

For any $0 \le n, m \le d$, we will denote by $N_n(C_d)$ (resp. $N^m(C_d)$) the $\R$-vector space of $n$-dimensional (resp. $m$-codimensional) cycles on $C_d$ modulo numerical equivalence. The intersection product 
induces a perfect duality $N^m(C_d)\times N_{d-m}(C_d)\stackrel{\cdot}{\longrightarrow} \R$. The vector space $N^*(C_d)=\oplus_{m=0}^d N^m(C_d)$ is a graded $\R$-algebra with respect to the intersection product. 

The \emph{tautological ring} $R^*(C_d)$ is the graded $\R$-subalgebra of $N^*(C_d)$ generated by the codimension one classes $\theta=\alpha_d^*([\Theta])$ (where $\Theta$ is any theta divisor on $J(C)$) and $x=[X_{p_0}]$ 
for some (equivalently any) base point $p_0$. 
Observe that $\theta$ is a semiample class (because it is the pull-back of an ample line bundle via a regular morphism) and it is ample if and only if $\alpha_d$ is a finite morphism, 
that is \color{black}if and only if $d < \gon(C)$. On the other hand, since we can move the base point $p_0$ arbitrarily, the class $x$ is ample by the Nakai-Moisezhon criterion (see  \cite[Prop. VII.2.2]{GAC1}).

We recall from \cite{BKLV} the following properties of the tautological ring $R^*(C_d)$ and vector spaces $R^m(C_d)$.

\begin{prop}\label{basetaut}
\noindent 
\begin{enumerate}[(i)]
\item \label{basetaut1} We have that $\theta^{g+1}=0$ and, if $s \in \bbN: s \le d$, 
$$\theta^s \cdot x^{d-s}=s!\binom{g}{s}=
\begin{cases} \frac{g!}{(g-s)!} & \text{ if }0\leq s\le g,\cr
 \ \ \ 0 & \text{ if } s > g. \cr
\end{cases}
$$
\item \label{basetaut4} For any $0\leq m\leq d$, set $r(m):=\min\{m,d-m, g\}$. Then
$R^m(C_d)$ has dimension $r(m)+1$ and it is freely generated by any subset of $r(m)+1$ monomials belonging to $\{x^m,x^{m-1}\theta, \ldots, x^{m-\min\{m,g\}} \theta^{\min\{m,g\}}\}$.

In particular, the monomials $\{x^m,\ldots, x^{m-r(m)}\theta^{r(m)}\}$ form a basis of $R^m(C_d)$, which is called the \emph{standard basis}. 
\item \label{basetaut5} The intersection product $R^m(C_d)\times R^{d-m}(C_d)\to \R$ is non-degenerate. 
\end{enumerate} 
\end{prop}
\begin{proof}
See 
\cite[Lemma 2.2 and Prop. 2.4]{BKLV}.
\end{proof}

\subsection{Cones of cycles}\label{cones}

Let us introduce the cones of cycles we will be working with. Inside the real vector space $N^m(C_d), 0 \le m \le d$, consider the (convex) cone $\Eff^m(C_d)$ generated by effective codimension $m$ cycles (called the cone of \emph{effective cycles}) and its closure $\Pseff^m(C_d)$ (called the cone of \emph{pseudoeffective cycles}). These cones are salient by \cite[Prop. 1.3]{BFJ}, \cite[Thm. 1.4(i)]{fl1}. The intersection $\Efft^m(C_d):=\Eff^m(C_d)\cap R^m(C_d)$ is called the \emph{tautological effective} cone and its closure  $\Psefft^m(C_d):=\ov{\Efft^m(C_d)}$ is called the \emph{tautological pseudoeffective} cone. Note that there is an inclusion $\Psefft^m(C_d)\subseteq \Pseff^m(C_d)\cap R^m(C_d)$, which a priori could be strict. 

The dual of $\Pseff^{d-m}(C_d)$ (respectively of $\Psefft^{d-m}(C_d)$) cone is  the \emph{nef} cone $\Nef^m(C_d)\subset N^m(C_d)$ (resp. the \emph{tautological nef} cone $\Neft^m(C_d)\subset R^m(C_d)$). Note that there is an inclusion $\Nef^m(C_d)\cap R^m(C_d)\subseteq \Neft^m(C_d)$,  which a priori could be strict. 

For $0 \le n \le d$ we set $\Eff_n(C_d):=\Eff^{d-n}(C_d)$ and similarly for the other cones. 

Note that, if $C$ is a very general curve, 
then $R^*(C_d)=N^*(C_d)$ for every $d\geq 1$ by 
\cite[Fact 2.6]{BKLV}, \cite[VIII.5]{GAC1}\color{black}, and hence $\Efft^m(C_d)=\Eff^m(C_d)$,  $\Psefft^m(C_d)= \Pseff^m(C_d)$ and $\Nef^m(C_d)=\Neft^m(C_d)$.

A  case where we know a complete description of the (tautological) effective, pseudoeffective and nef cone of cycles is the case of curves of genus one\footnote{Note that if $g=0$, i.e. $C=\PP^1$, then 
$C_d \cong \PP^d$ and all the cones in question become one-dimensional, hence trivial.}. 

\begin{example}[Genus $1$ - {\cite[Example 2.9]{BKLV}}\label{genus1}]
If the curve $C$ has genus $1$, then for any $1\leq m \leq d-1$ we have that $N^m(C_d)=R^m(C_d)$ and 
\begin{equation*}\label{E:cones-g1}
\Pseff^m(C_d)=\Nef^m(C_d)=\cone\left(x^{m-1}\theta, x^m-\frac{m}{d}x^{m-1}\theta\right) \subset N^m(C_d)_{\R} \cong \R^2.
\end{equation*}
\end{example}
\noindent This follows either by \cite{Ful} or by \cite[Theorem B]{BKLV} and Theorem \ref{T:B}\eqref{T:B1}.

\section{Abel-Jacobi faces}\label{S:AJ}

The aim of this section is to introduce some faces of the (tautological or not) pseudoeffective cones of $C_d$ obtained as contractibility faces of the Abel-Jacobi morphism $\alpha_d:C_d\to \Pic^d(C)$. 

\subsection{Contractibility faces}\label{SS:contract}

In this subsection,  we will  introduce the contractibility faces associated to any morphism $\pi:X\to Y$ between irreducible projective varieties. 
The definition of the contractibility faces is based on the contractibility index introduced in \cite[\S 4.2]{fl2}. 

\begin{defi}\label{contract}
Let $\pi:X \to Y$ be a morphism between irreducible projective varieties and fix the class $h\in N^1(Y)$ of an ample Cartier divisor on $Y$. Given an element $\alpha\in \Pseff_k(X)$ for some $0\leq k \leq \dim X$, the \emph{contractibility index} of $\alpha$, denoted by $\contr_{\pi}(\alpha)$, is equal to the largest non-negative 
integer \color{black}$c\leq k+1$ such that $\alpha\cdot \pi^*(h^{k+1-c})=0$.
\end{defi}
Since $\alpha\cdot \pi^*(h^{k+1})=0$ for dimension reasons, the  contractibility index is well-defined and it is easy to see that it does not depend on $h$. The following properties are immediate:
\begin{itemize}
\item $\max\{0, k-\dim \pi(X)\}\leq \contr_{\pi}(\alpha)\leq k+1$ and equality holds in the last inequality if and only if  $\alpha=0$;
\item $\contr_{\pi}(\alpha)>0 \Longleftrightarrow \pi_*(\alpha)=0$;
\item If $\alpha=[Z]$ for an irreducible subvariety $Z\subseteq X$ of dimension $k$, then $\contr_{\pi}(Z):=\contr_{\pi}(\alpha)=\dim Z-\dim \pi(Z)$.
\end{itemize}
\begin{defi}\label{contract2}
Let $\pi:X \to Y$ be a morphism between irreducible projective varieties and let $k, r$ be integers such that $0 \le k \le \dim X, r \ge 0$. Set
$$F_k^{\geq r}(\pi) = \cone(\{\alpha \in \Pseff_k(X) : \contr_{\pi}(\alpha)\ge r \}).$$  
We set $c_{\pi}(r) = -1$ if there is no irreducible subvariety $Z \subseteq X$ with $\contr_{\pi}(Z)\ge r$; otherwise we define 
\begin{displaymath}
c_{\pi}(r) =\max\left\{ 0 \le k \le \dim X \left| \begin{array}{l} \text{there exists an irreducible subvariety }Z \subseteq X\\ \text{with } \dim Z = k \text{ and } \contr_{\pi}(Z)\ge r \end{array}\right.\right\}.
\end{displaymath}
\end{defi}
Note that $F_k^{\ge r}(\pi) = \emptyset$ if $r > k+1$, $F_k^{\ge k+1}(\pi) = \{0\}$,  $F_{\dim X}^{\ge r}(\pi) = \{0\}$ if and only if $r \ge 1 + \dim X -\dim \pi(X)$ and $F_k^{\ge r}(\pi) = \Pseff_k(X)$ if and only if $r \le \max\{0, k-\dim \pi(X)\}$. Moreover, if $c_{\pi}(r) \ge 0$, then $r \le c_{\pi}(r) \le \dim X$.


The following criterion of extremality follows from \cite[Thm. 4.15]{fl2} and it is an improvement of \cite[Prop. 2.1, 2.2 and Rmk. 2.7]{CC}. 

\begin{prop}\label{extr-crit}
Let $\pi:X \to Y$ be a morphism between projective irreducible varieties and fix $k, r$ integers such that $1+ \max\{0, k-\dim \pi(X)\} \le r \le k \le \dim X$. 
Then
\begin{enumerate}[(i)]
\item \label{extr-crit1} The cone $F_k^{\geq r}(\pi)$ is a face of $\Pseff_k(X)$. In particular, the cone $F_k^{\geq r}(\pi) \cap \Eff_k(X)$ is a face of $ \Eff_k(X)$. Moreover $F_k^{\geq r}(\pi)$ is non-trivial
for $r \le k \le c_{\pi}(r)$.
\item \label{extr-crit2} Suppose that $r \le k \le c_{\pi}(r)$. The  number of irreducible subvarieties of $X$ of dimension $k$ and contractibility index at least $r$ is finite if and only if $k=c_{\pi}(r)$.

In this case, if we denote by  $Z_1,\ldots, Z_s$ the irreducible subvarieties of $X$ of dimension $c_{\pi}(r)$ and contractibility index at least $r$, we have that
$$F_{c_{\pi}(r)}^{\geq r}(\pi)=\cone([Z_1],\ldots,[Z_s])=F_{c_{\pi}(r)}^{\geq r}(\pi)\cap \Eff_k(X).$$
\end{enumerate}
\end{prop}
Because of \eqref{extr-crit1}, we will call $F_k^{\geq r}(\pi)$ the \emph{$r$-th contractibility face} of $\Pseff_k(X)$. 
\begin{proof}
Note that for any $\alpha \in \Pseff_k(X)$, we have $\alpha \in  F_k^{\geq r}(\pi)$ if and only if $\alpha\cdot \pi^*(h^{k+1-r})=0$. Let $\beta_1, \beta_2 \in \Pseff_k(X)$ be such that $\beta_1+\beta_2 \in F_k^{\geq r}(\pi)$. Then 
$\beta_1\cdot \pi^*(h^{k+1-r})+\beta_2\cdot \pi^*(h^{k+1-r}) = 0$ and $\beta_i\cdot \pi^*(h^{k+1-r}) \in \Pseff_{r-1}(X)$ (because $\pi^*(h)$ is nef, hence limit of ample classes) for $i=1, 2$, so that $\beta_1\cdot \pi^*(h^{k+1-r}) = \beta_2\cdot \pi^*(h^{k+1-r})=0$ since $\Pseff_{r-1}(X)$ is salient by  \cite[Prop. 1.3]{BFJ}, \cite[Thm. 1.4(i)]{fl1}. Then $\beta_1, \beta_2 \in F_k^{\geq r}(\pi)$. This proves the first assertion in \eqref{extr-crit1}. 

Assume now that $r < k \le c_{\pi}(r)$ and let $Z \subseteq X$ be an irreducible subvariety such that $\dim Z = k, \contr_{\pi}(Z) \ge r$. We claim that there are infinitely many irreducible subvarieties $W \subseteq X$ with $\dim W = k-1$ and $\contr_{\pi}(W) \ge r$. It follows by this claim that $F_k^{\geq r}(\pi)$ is non-trivial
for $r \le k \le c_{\pi}(r)$. To see the claim we consider two cases. If $\pi(Z)$ is not a point, then pick a generic codimension one subvariety $V \subset \pi(Z)$ such that $V$ intersects the open subset of $\pi(Z)$ where fibers of $\pi_{|Z}$ have dimension $\contr_{\pi}(Z)$. The inverse image $(\pi_{|Z})^{-1}(V)$ will have an irreducible component $W$ that dominates $V$ and therefore with $\dim W = k-1$ and $\contr_{\pi}(W) = \contr_{\pi}(Z) \ge r$. If $\pi(Z)$ is a point, then pick any codimension one subvariety $W \subset Z$. Then $\dim W = k-1$ and $\contr_{\pi}(W) = \dim W = k-1 \ge r$. In either case, there are infinitely many such 
subvarieties $W$ \color{black}and the claim is proved.

Consider now the first assertion of part \eqref{extr-crit2}. The only if part follows immediately by the claim starting with an irreducible subvariety $Z \subset X$ of dimension $c_{\pi}(r) $ and contractibility index at least $r$ (such a $Z$ exists by the definition of $c_{\pi}(r)$). The if part is proved in \cite[Thm. 4.15(1)]{fl2}. 

For the second assertion of \eqref{extr-crit2}: the first equality follows from \cite[Thm. 4.15(2)]{fl2} and the second equality follows directly from the first one.
\end{proof}

\begin{remark}\label{cong}
Let $1 \le r \le k \le \dim X$. It is natural to wonder if the following statements hold true:
\begin{enumerate}
\item (Strong$^r(\pi)$) $F_k^r(\pi)=\ov{F_k^r(\pi)\cap \Eff_k(X)}$.
 \item (Weak$^r(\pi)$) $\langle F_k^r(\pi)\rangle=\langle F_k^r(\pi)\cap \Eff_k(X)\rangle$.
\end{enumerate}
For $r=1$, the above statements reduce to, respectively,  the strong and weak conjecture in \cite[Conj. 1.1]{fl2}. If Weak$^{r}(\pi)$ holds true then 
$F_k^{\ge r}(\pi) = \{0\}$ for any $k>c_{\pi}(r)$. 
 If we also assume that $k \ge r \ge \dim X - \dim \pi(X) + 1$ it is easy to see, using \cite[Thm. 4.13]{fl2}, that the last expectation holds. Moreover, since $F_k^{\ge r}(\pi) \subseteq F_k^{\ge r-1}(\pi)$ we expect that $F_k^{\ge r}(\pi) = \{0\}$ when $k > c_{\pi}(1)$.
\end{remark}

\subsection{Brill-Noether varieties}\label{SS:BNvar}

In order to apply the previous criterion to  the 
Abel-Jacobi map $\alpha_d: C_d \to \Pic^d(C)$, we need to know the subvarieties of $C_d$ that have contractibility index at least $r$ with respect to $\alpha_d$.
As we will see, these subvarieties turn out to be contained in the \emph{Brill-Noether} variety $C_d^r\subseteq C_d$ which is defined (set theoretically) as:
$$C_d^r:=\{D\in C_d: \: \dim |D|\geq r\}.$$
Note that $C_d^r=\alpha_d^{-1}(W_d^r(C))$ where $W_d^r(C)$ is the Brill-Noether variety in $\Pic^d(C)$ which is defined (set theoretically) as
$$W_d^r(C) = \{L\in \Pic^d(C): \: h^0(C,L)\geq r+1\}.$$
The Brill-Noether varieties $C_d^r$ and $W_d^r(C)$ are in a natural way determinantal varieties (see \cite[Chap. IV]{GAC1}).
From Riemann-Roch theorem, we have the following trivial cases for $W_d^r(C)$ and $C_d^r$:
\begin{itemize}
\item If $r\leq \max\{-1,d-g\}$ then $W_d^r(C)=\Pic^d(C)$, and hence $C_d^r=C_d$. 
\item If $r=0$ and $d\leq g-1$ then $\alpha_d: C_d^0=C_d\to W_d^0(C)$ is a resolution of singularities. 
\item If $d\geq 2g-1$ then $W_d^r(C)=
\begin{cases}
\Pic^d(C) & \text{ if } r\leq d-g, \\
\emptyset & \text{ if } r>d-g,
\end{cases}$
and 
$C_d^r=
\begin{cases}
C_d & \text{ if } r\leq d-g, \\
\emptyset & \text{ if } r>d-g.
\end{cases}$
\end{itemize}

The non-emptiness of $C_d^r$ is equivalent to the existence of a linear system of degree $d$ and dimension $r$ on $C$, and we define an invariant of $C$ controlling the existence of such linear systems. 


\begin{defi}\label{D:gonindex}
For any integer $r \ge 1$, the \emph{$r$-th gonality index} of $C$, denoted by  $\gon_r(C)$, is the smallest integer $d$ such that $C$ admits a $\g_d^r$.  
\end{defi}
Clearly, $d\geq \gon_r(C)$ if and only if the curve $C$ admits  a $\g_d^r$. Observe that if $r=1$ then $\gon_1(C)$ is the (usual) gonality $\gon(C)$ of $C$. 
The possible values that the $r$-th gonality index can achieve are described in the following

\begin{lemma}\label{L:gonindex}
The $r$-th gonality index of $C$ satisfies the following
\begin{equation}\label{gonlarge}
\gon_r(C)=
\begin{cases}
g+r & \text{ if }r\geq g,\\
2g-2 & \text{ if }r=g-1, 
\end{cases}
\end{equation}
\begin{equation}\label{gonsmall}
2r\leq \gon_r(C)\leq \gamma(r):=
\left\lceil \frac{rg}{r+1} \right\rceil +r  \: \text{if }1\leq r \leq g-2,
\end{equation}
where the first inequality is achieved if and only if $C$ is hyperelliptic and the second inequality is achieved  for the general curve $C$.
\end{lemma}
\begin{proof}
From Clifford's inequality and Riemann-Roch theorem, it follows easily that:
\begin{itemize}
\item any $\g_d^{g-1}$ on $C$ is such that $d\geq 2g-2$ with equality if and only if $\g_{2g-2}^{g-1}$ is the 
 complete canonical system $|K_C|$;
\item  if $r\geq g$ then any $\g_d^{r}$ is such that $d\geq r+g\geq 2g$.
\end{itemize}
These two facts imply the first part of the statement.

For the second part of the statement: the lower bound is provided by Clifford's theorem and we have equality if and only if the curve is hyperelliptic; the upper bound is provided by Brill-Noether theory and equality holds for the general curve by 
\cite[Chap. V, Thm. 1.5]{GAC1} (the proof of Griffiths and Harris works over any algebraically closed field, see \cite{oss}) .
\end{proof}

\begin{remark}\label{R:r-gon}
It follows easily from the previous lemma that if $d\geq \gon_n(C)$ then $d-n\geq \min\{n,g\}$, or equivalently that $r(n):=\min\{n,d-n,g\}=\min\{n,g\}.$
\end{remark}

The properties of the Brill-Noether varieties (in the non-trivial cases) are collected in the following fact that summarizes the main results of the classical Brill-Noether theory (see \cite[Chap. IV and VII]{GAC1}). 

\begin{fact}\label{BNclass}
Fix integers $r$ and $d$ such that $\max\{1,d-g+1\}\leq r$ and $0 \leq d\leq 2g-2$. 
\begin{enumerate}[(i)]
\item \label{BNclass0} The open subset $C_d^r\setminus C_d^{r+1}\subset C_d^r$ is dense. Therefore, the morphism $(\alpha_d)_{|C_d^r}:C_d^r \twoheadrightarrow W_d^r$ is generically a $\PP^r$-fibration and each irreducible component of $C_d^r$ has contractibility index exactly $r$. 
\item \label{BNclass00} $C_d^r$ is non-empty if and only if $d\geq \gon_r(C)$. In particular, we have the following 
\begin{equation}\label{E:C-nonemp}
d\geq \gamma(r):= \left\lceil \frac{rg}{r+1} \right\rceil +r \Rightarrow C_d^r\neq \emptyset \Rightarrow d\geq 2r.
\end{equation}
\item \label{BNclass1} 
If $C_d^r$ is non-empty, every irreducible component of $C_d^{r}$ has dimension at least $r+\rho(g,r,d)=r+g-(r+1)(g-d+r)=d-r(g-d+r)$ and at most $r+(d-2r)=d-r$.
\item \label{BNclass2} Assume that either $C_d^{r}$ is empty or has pure dimension $r+\rho(g,r,d)$. Then the class of $C_d^{r}$  is equal to 
$$[C_d^{r}]=c_d^r:=\prod_{i=0}^r \frac{i!}{(g-d+r+i-1)!} \sum_{\alpha=0}^r(-1)^{\alpha} \frac{(g-d+r+\alpha-1)!}{\alpha!(r-\alpha)!} x^{\alpha}\theta^{r(g-d+r)-\alpha}.$$
\item \label{BNclass3} Assume that $C$ is a general curve of genus $g$.
\begin{itemize}
\item If $\rho(g,r,d)<0$ then $C_d^r$ is empty;
\item If $\rho(g,r,d)=0$ then $C_d^r$ is a disjoint union of $g!\displaystyle \prod_{i=1}^r\frac{i!}{(g-d+r+i)!}$ projective spaces of dimension $r$;
\item If $\rho(g,r,d)>0$ then $C_d^r$ is irreducible of dimension $r+\rho(g,r,d)$.  
\end{itemize}
\end{enumerate}
\end{fact}
A curve satisfying the conditions of \eqref{BNclass3} is called a \emph{Brill-Noether general} curve.
\begin{proof}
\eqref{BNclass0}:  the first assertion follows from the fact that  there are no irreducible components of $C_d^r$ contained in $C_d^{r+1}$ by 
\cite[Chap. IV, Lemma 1.7]{GAC1} (the proof works over any algebraically closed field).
Using that the restriction of $\alpha_d$ to $C_d^r\setminus C_d^{r+1}$ is a $\PP^r$-fibration, the remaining assertions follow. 
 
\eqref{BNclass00}: $C_d^r$ is non-empty if and only if there exists a $\g_d^r$ on $C$ which is equivalent to the condition $d\geq \gon_r(C)$. The chain of implications \eqref{E:C-nonemp} follows then from Lemma \ref{L:gonindex}. 

Using \eqref{BNclass0}, part \eqref{BNclass1} follows from the fact that every irreducible component of $W_d^r(C)$ has dimension greater or equal to $\rho(g,r,d)$ by \cite[Chap. IV, Lemma 3.3]{GAC1}, 
\cite{kl1, kl2} and dimension at most $d-2r$ by Martens' theorem 
(see \cite[Chap. IV, Thm. 5.1]{GAC1}, \cite[Thm. 1]{mar})


For part \eqref{BNclass2}, see 
\cite[Chap. VII, \S 5]{GAC1} (the proof works over any algebraically closed field).

Part \eqref{BNclass3}: we will distinguish three cases according to the sign of $\rho(g,r,d)$. If $\rho(g,r,d)<0$ then $W_d^r(C)$ is empty by \cite[Chap. V, Thm. 1.5]{GAC1} 
(the proof of Griffiths and Harris works over any algebraically closed field, see \cite{oss}) and hence also $C_d^r$ is empty. If $\rho(g,r,d)=0$ then $W_d^r(C)$ consists of  finitely many $\g_d^r$ 
(see \cite[Chap. V, Thm. 1.3 and 1.6]{GAC1} - again holding over any algebraically closed field), whose number is equal to $g!\displaystyle \prod_{i=1}^r\frac{i!}{(g-d+r+i)!}$  by Castelnuovo's formula \cite[Chap. V, Formula (1.2)]{GAC1}, \cite{kl2}; hence the result for $C_d^r$ follows. If $\rho(g,r,d)>0$ then $W_d^r(C)$ is irreducible of dimension $\rho(g,r,d)$ by \cite[Chap. V, Thm. 1.4, Cor. of Thm. 1.6]{GAC1} 
and by \cite{gies}, \cite[Thm. 1.1 and Rmk. 1.7]{fl}, from which we deduce that $C_d^r$ is irreducible of dimension $r+\dim W_d^r(C)=r+\rho(g,r,d)$ using \eqref{BNclass0}.
\end{proof}

There are some  Brill-Noether varieties that are pure of the expected dimension for any curve (and not only for the general curve), as described in the following example. 

\begin{example}\label{BN=subor}
For any $g\leq d \leq 2g-2$, the Brill-Noether variety $C_d^{d-g+1}$ is irreducible of the expected dimension $g-1$ \footnote{Indeed, these are the unique Brill-Noether varieties that are also subordinate varieties; more specifically, $C_d^{d-g+1}=\Gamma_d(|K_C|)$, with the notation of  \eqref{D:sublocus}.}.  Indeed, the variety $W_d^{d-g+1}(C)$ is irreducible of dimension $2g-2-d$ since it is isomorphic, via the residuation map $L\mapsto K_C\otimes L^{-1}$, to the variety $W_{2g-2-d}^0(C)=\Im(\alpha_{2g-2-d})$. 
We conclude that $C_d^{d-g+1}$ is irreducible of dimension $d-g+1+\dim W_d^{d-g+1}(C)=g-1$ by Fact \ref{BNclass}\eqref{BNclass0}. 

Therefore, Fact \ref{BNclass}\eqref{BNclass2} implies that the class of $C_d^{d-g+1}$ is equal to 
\begin{equation}\label{E:clBN}
[C_d^{d-g+1}]= \sum_{\alpha=0}^{d-g+1} (-1)^{\alpha} \frac{x^{\alpha}\theta^{d-g+1-\alpha}}{(d-g+1-\alpha)!}.
\end{equation} 
\end{example}

\subsection{Abel-Jacobi faces}\label{SS:AJfaces}

We can now study the contractibility faces  associated to the 
Abel-Jacobi morphism $\alpha_d:C_d\to \Pic^d(C)$.

\begin{defi}[Abel-Jacobi faces]\label{D:AJfaces}
Let $0 \leq n \leq d$. For any $r$ such that $1 + \max\{0,n-g\}= 1 + \max\{0,n-\dim \alpha_d(C_d)\}\leq r\leq n$, let $\AJ_n^r(C_d):=F_n^{\geq r}(\alpha_d)\subseteq \Pseff_n(C_d)$ and call it the $r$-th \emph{Abel-Jacobi face} in dimension $n$. Moreover, we set $\AJt_n^r(C_d):=F_n^{\geq r}(\alpha_d)\cap \Psefft_n(C_d)\subseteq \Psefft_n(C_d)$ 
and call it the $r$-th \emph{tautological Abel-Jacobi face} in dimension $n$. 
\end{defi}


\begin{remark}
\label{nontr}
Let $0 \le n \leq d, 1 + \max\{0,n-g\} \le r \le n$. If $d < \gon(C)$ then $\theta$ is ample, whence $\AJ^r_n(C_d) = \AJt^r_n(C_d) = \{0\}$ by \cite[Cor. 3.15]{fl1}, \cite[Prop. 3.7]{fl2}. 
\end{remark}

Applying Proposition  \ref{extr-crit} to our case, we get the following result  that guarantees that the Abel-Jacobi faces are non-trivial, under suitable assumptions. 

\begin{prop}\label{P:nonzero}
Let $1 \le n \le d-1$ and let $1 + \max\{0,n-g\} \le r \le n$. Then
\begin{equation}
\label{cpi}
c_{\alpha_d}(r) = \begin{cases} -1 & \text{ if } d < \gon_r(C) \ \mbox{(or equivalently} \ C_d^r = \emptyset)\cr \dim C_d^r & \text{ if } d \ge \gon_r(C)  \ \mbox{(or equivalently} \ C_d^r \neq \emptyset). \cr \end{cases}
\end{equation}
Moreover $\AJ_n^r(C_d) = \{0\}$ whenever $1 + \max\{0,d-g\} \le r \le n$ and either $d < \gon_r(C)$ or $d \ge \gon_r(C)$ and $n > \dim C_d^r$.

Assume now that $d\geq \gon_r(C)$ (which then forces $\dim C_d^r\geq r$). Then the following hold:
\begin{enumerate}[(i)]

\item \label{BNextr1} $\AJ_n^r(C_d)$ is non-trivial if $n \le \dim C_d^r$.

\item \label{BNextr2} $\AJ_{\dim C_d^r}^r(C_d)$ is equal to $\AJ_{\dim C_d^r}^r(C_d)\cap \Eff_n(C_d)$ and it is the conic hull  of the irreducible components of $C_d^r$ of maximal dimension.
\end{enumerate} 
Furthermore, \eqref{BNextr1} holds for $\AJt_n^r(C_d)$ if $C_d^r$ has some tautological irreducible component of maximal dimension 
and \eqref{BNextr2}  holds for $\AJt_n^r(C_d)$  if all irreducible components of $C_d^r$  of maximal dimension are tautological. 
\end{prop}
\begin{proof}
We will apply Proposition \ref{extr-crit} to the Abel-Jacobi map $\alpha_d:C_d\to \Pic^d(C)$.

Observe that if there is an irreducible subvariety $Z \subseteq C_d$ of  contractibility index at least $r$ then $C_d^r \neq \emptyset$ and $Z \subseteq C_d^r$. Moreover we claim that each irreducible component of $C_d^r$ has contractibility index at least $r$. In fact, if $r > \max\{0,d-g\}$ the claim follows by Fact \ref{BNclass}\eqref{BNclass0} while if $r \le \max\{0,d-g\}$ then $C^r_d = C_d$ that has contractibility index $\max\{0,d-g\} \ge r$.

This proves \eqref{cpi} and, if $C_d^r \neq \emptyset$, that the subvarieties of dimension $c_{\alpha_d}(r)$ and contractibility index at least $r$ are exactly the irreducible components of $C_d^r$ of maximal dimension. Using these facts, the first part of the proposition follows from Remark \ref{cong} and Proposition \ref{extr-crit}. 

In order to prove the same properties for $\AJt_n^r(C_d)$, observe that the non-triviality of $\AJt_{\dim C_d^r}^r(C_d)$
and the analogue of \eqref{BNextr2} for $\AJt_{\dim C_d^r}^r(C_d)$, follow directly by our assumption. On the other hand, the non-triviality of $\AJt_n^r(C_d)$ for $n < \dim C_d^r$ follows from the proof of Proposition \ref{extr-crit} using that there is one tautological component of $C_d^r$ of dimension equal to $\dim C_d^r$.
\end{proof}

\begin{remark}\label{congAJ}
According to Remark \ref{cong}, we expect that, 
for any $1 + \max\{0,n-g\} \le r \le n$, $\AJ_n^r(C_d)=\{0\}$  if either $d<\gon_r(C)$ (which is equivalent to $C_d^r=\emptyset$) or $d\geq \gon_r(C)$ and $n>\dim C_d^r$. In case $r=1$, this would follow from the validity of the  weak conjecture 1.1 in \cite{fl2}  for the Abel-Jacobi morphism. Indeed, we know that $\alpha_d$ satisfies the 
above mentioned conjecture  if $d < \gon_1(C)$ (in which case it holds trivially) and  if $d \ge g$ and the (algebraically closed) base field is  uncountable,  by \cite[Thm. 1.2]{fl3}.
\end{remark}
As a corollary of the above proposition,  we can determine some ranges of $d$ and $n$ for which we can find non-trivial Abel-Jacobi faces in $\Pseff_n(C_d)$ or $\Psefft_n(C_d)$.

\begin{cor}\label{C:nontriv}
Let $1 \le n \le d-1$ and let $C$ be a curve of genus $g\geq 1$.
\begin{enumerate}[(i)]
\item \label{P:nontriv1} There exist non-trivial Abel-Jacobi faces of $\Pseff_n(C_d)$ if $d \ge \frac{n+g+1}{2}$. 
\item  \label{P:nontriv2} There exist non-trivial Abel-Jacobi faces of $\Psefft_n(C_d)$ if either $d \ge g+1$ or $d \ge \frac{n+g+1}{2}$ and $C_d^1$ has some tautological irreducible component of maximal dimension (which holds true  if $C$ is a Brill-Noether general curve).
\end{enumerate}
\end{cor}

Note that the inequality $d\geq \frac{n+g+1}{2}$ automatically holds if either $n \ge g-1$ or $d-n\ge \frac{g}{2}$. 
From the proof of the corollary, it will  follow that the lower bound $d\geq \frac{n+g+1}{2}$ is sharp for Brill-Noether general curves provided that the expectation of Remark \ref{congAJ} holds true. 
On the other hand, for special curves, the lower bound is far from being sharp, see Theorem \ref{T:hyper} for the case of hyperelliptic curves.
 
\begin{proof}
We will distinguish three cases.

\begin{itemize}
\item If $g\leq n$ (which implies that $g+1\leq d$) then $C_d^{n-g+1}=C_d$ by Riemann-Roch, and hence Proposition \ref{P:nonzero}\eqref{BNextr1} implies that $\AJ_n^{n-g+1}(C_d)$ and $\AJt_n^{n-g+1}(C_d)$ are non-trivial.

\item If $n\leq g\leq d-1$ then $C_d^1=C_d$ by Riemann-Roch, and hence Proposition \ref{P:nonzero}\eqref{BNextr1} implies that $\AJ_n^{1}(C_d)$ and $\AJt_n^{1}(C_d)$ are non-trivial.

\item If $d\leq g$ (which implies that $n\leq g-1$) then Fact \ref{BNclass}\eqref{BNclass1} gives that $\dim C_d^1\geq 2d-g-1$
if $C_d^1 \neq \emptyset$. Hence, if $n\leq 2d-g-1$ then Proposition \ref{P:nonzero}\eqref{BNextr1} 
and Fact \ref{BNclass}\eqref{BNclass00} imply that $\AJ_n^1(C_d)$ is non-trivial and, furthermore, that $\AJt_n^1(C_d)$ is non-trivial provided that  $C_d^1$ has some tautological irreducible component of maximal dimension.
\end{itemize}
\end{proof}

\begin{remark}\label{R:Kummer}
One may wonder if one could get more faces of the pseudoeffective cone of $C_d$ by looking at contractibility faces of some other  regular morphism  $f\colon C_d\to Z$ to some projective variety. There is no loss of generality (using the Stein factorization) in assuming that $f$ is a regular fibration, 
i.e. $f_*(\O_{C_d})=\O_Z$. 
Any such regular fibration  is uniquely determined (up to isomorphism) by $f^*(\Amp(Z))$ which is a face of the semiample cone of $C_d$. 

The intersection of the semiample cone with $R^1(C_d)$ is a subcone of the two dimensional cone $\Neft^1(C_d)$ which has two extremal rays: one is spanned by $\eta_{1,d}=dgx-\theta$ which is the dual of the class of the small diagonal $\Delta_{(d)}$ (see \cite[Cor. 3.15]{BKLV})  and the other one is generated by $\theta$ provided that $d\geq \gon(C)$ (see Theorem \ref{thetaNef}). The 
Abel-Jacobi morphism $\alpha_d:C_d\twoheadrightarrow \alpha_d(C_d)\subseteq \Pic^d(C)$ corresponds to the face $\cone(\theta)$ while the other face $\cone(\eta_{1,n})$ corresponds to another fibration that we are going to describe.

Consider the regular morphism  (as in \cite[\S 2.2]{Pac})
$$
\begin{aligned}
\phi_d: C^d & \longrightarrow J(C)^{{d \choose 2}}\\
 (p_1,\ldots, p_d) & \mapsto (\O_C(p_i-p_j))_{1 \le i < j \le d}.
 \end{aligned}
 $$
By quotienting $C^d$ by the symmetric group $S_d$ and $J(C)^{{d \choose 2}}$ by the semi-direct product $\Z/2\Z^{{d \choose 2}}\rtimes S_{{d \choose 2}}$ (where $S_{{d \choose 2}}$ acts by permutation and each copy of $\Z/2\Z$ acts on the corresponding factor $J(C)$ as the inverse), we get a regular 
fibration 
\begin{equation}\label{E:Kum}
\varphi_d : C_d \twoheadrightarrow \varphi_d(C)\subset \Sym^{{d \choose 2}}(\Kum(C)).
\end{equation}
It is easily checked that  the only subvariety contracted by $\varphi_d$ is $\Delta_{(d)}$. We then have $c_{\varphi_d}(r) = -1$ if $r \ge 2$ and $c_{\varphi_d}(1) = 1$. By Proposition \ref{extr-crit}(ii) we get that $\cone(\Delta_{(d)})$ is an extremal ray of $\Pseff_1(C_d)$ (which improves \cite[Lemma 2.2]{Pac} where the author uses the above maps to show that the class of the small diagonal $\Delta_{(d)}$ lies in the boundary of $\Pseff_1(C_d)$). In fact we know more, namely that $\cone(\Delta_{(d)})$ is an edge of $\Pseff_1(C_d)$ by \cite[Cor. 3.15(d)]{BKLV}. According to Remark \ref{cong}, we also expect that $F_k^{\ge r}(\varphi_d) = \{0\}$ for $k \ge 2$ or $k=1$ and $r \ge 2$. Hence, we do not expect to find new interesting faces by looking at the contractibility faces of $\phi_d$, apart from a new (and simpler) proof of the fact that $\Delta_{(d)}$ spans an extremal ray of $\Psefft_1(C_d)$. 
\end{remark}

\subsection{Brill-Noether rays}\label{SS:BNrays}

In this subsection, we use Proposition \ref{P:nonzero} to exhibit some extremal rays of $\Pseff_n(C_d)$ (and of $\Psefft_n(C_d)$) for a 
Brill-Noether \color{black}general curve.

\begin{thm}\label{AJtaut}
Let $\max\{1,d-g+1\}\leq  r$ and $\displaystyle  \frac{rg}{r+1}+r\leq d\leq 2g-2$.  
Assume that $C$ is a Brill-Noether general curve.

Then $\AJ^r_{r+\rho}(C_d)=\AJt^r_{r+\rho}(C_d)=\cone([C_d^r])$,  where $\rho:=\rho(g,r,d)=g-(r+1)(g-d+r)$. In particular, $[C_d^r]$ generates an extremal ray (called the \emph{BN(=Brill-Noether) ray}) of $\Pseff_{r+\rho}(C_d)$ and of $\Psefft_{r+\rho}(C_d)$. 
\end{thm}
 
Note that if $r=1$ and $\frac{g+2}{2}\leq d \leq g$ then $[C_d^1]$ generates an extremal ray of $\Psefft_{2d-g-1}(C_d)$, and this achieves the lower bound on $d$ in Corollary \ref{C:nontriv}.  

\begin{proof}
This will follow from  Proposition \ref{P:nonzero}\eqref{BNextr2} and its analogue  for the tautological Abel-Jacobi faces, provided that we  show that either $C_d^r$ is  tautological and irreducible of dimension $r+\rho$ or all the irreducible components of $C_d^r$ are tautological, of  dimension $r+\rho$ and numerically equivalent (in which case the class of $C_d^ r$ is a positive multiple of the class of each of its irreducible components).  

The hypothesis $ \frac{rg}{r+1}+r\leq d$ is equivalent to $\gon_r(C)\leq d$ by Lemma \ref{L:gonindex} (which is in turn equivalent to $C_d^r\neq \emptyset$) and it implies that $C_d^r$ has pure dimension $r+\rho$ by Fact \ref{BNclass}\eqref{BNclass3} and it has tautological class by Fact \ref{BNclass}\eqref{BNclass2}. We now distinguish two cases, according to the sign of $\rho$. If $\rho>0$ then $C_d^r$ is irreducible by Fact \ref{BNclass}\eqref{BNclass3} and we are done. 
If instead $\rho=0$, then $C_d^r$ is a disjoint union of $r$-dimensional fibers of the map $\alpha_d$  by Fact \ref{BNclass}\eqref{BNclass3}.  We conclude by observing that all the $r$-dimensional fibers of $\alpha_d$ are numerically equivalent  and they have tautological class  (indeed, their class is equal to $\Gamma_d(\g^r_d)$, see Fact \ref{subcyc}). 
\end{proof}

\begin{example}\label{BNexa}
Two special cases of BN rays of fixed codimension $m$ (which are also the unique ones in codimension $m$ if $m$ is a prime)  are the ones generated by the following Brill-Noether varieties:
\begin{enumerate}[(i)]
\item \label{BNexa1} If $1\leq m\leq g/2$ and $C$ is a 
Brill-Noether general curve, then $C_{g-m+1}^1$ is a pure codimension $m$ (and irreducible if and only if $m<g/2$ or $g=2$) subvariety of $C_{g-m+1}$ of class
$$[C_{g-m+1}^1]=\frac{\theta^m}{m!}-\frac{x\theta^{m-1}}{(m-1)!}.$$
\item \label{BNexa2} If $1\leq m\leq g-1$ (and $C$ is any curve) then $C_{g+m-1}^m$ is a codimension $m$ irreducible subvariety of $C_{g+m-1}$ of class (see Example \ref{BN=subor}) 
$$[C_{g+m-1}^m]=\sum_{\alpha=0}^m (-1)^{\alpha} \frac{x^{\alpha}\theta^{m-\alpha}}{(m-\alpha)!}.$$
\end{enumerate}

If $m=1$ in each of the above special cases, we get that $[C_{g}^1]=\theta-x\in N^1(C_g)$ generates an extremal ray of $\Pseff^1(C_{g})$, thus extending \cite[Rmk 1 after Thm.\ 5]{Kou} from very general curves to  arbitrary curves. 

\end{example}

It is natural to ask if  BN rays are perfect, i.e. if they are edges, in the entire or tautological pseudoeffective. As we will see, a way to prove this for the tautological pseudoeffective cone would be to apply Proposition \ref{P:AJtheta}\eqref{P:AJtheta2}. On the other hand we will show in Remark \ref{R:BNperf} that the unique BN rays $\cone([C_d^r])$ to which we can apply Proposition \ref{P:AJtheta}\eqref{P:AJtheta2}, and hence deduce that they are perfect rays, are those with $\rho=\rho(g,r,d)=0$ (when we will actually see in Remark \ref{R:compaAJ} that they coincide with the subordinate edge) and those with $d=g+r-1$ (when we will actually see in Theorem \ref{T:BNfaces} that they coincide with the  BN edge in dimension $g-1$).

\subsection{The $\theta$-filtration}\label{SS:theta}

The tautological Abel-Jacobi faces can be described in terms of a multiplicative filtration of the tautological ring $R^*(C_d)$, determined by the class $\theta$. 

\begin{defi}\label{thetalin}[The $\theta$-filtration]
For any  $0\leq m \leq d$ and any $0\leq i \leq g+1$,  let $\theta^{\geq i, m}$ (or simply $\theta^{\geq i}$ if $m$ is clear from the context) be the smallest linear subspace of $R^m(C_d)=R_{d-m}(C_d)$ 
containing the monomials $\{\theta^i x^{m-i}, \theta^{i+1}x^{m-i-1}, \ldots, \theta^{m} \}$, 
with the obvious convention that $\theta^{\geq i, m}=\{0\}$ if $i>m$. 
\end{defi}

The subspaces $\{\theta^{\geq i, m}\}$ form an exhaustive  decreasing multiplicative filtration of the tautological ring $R^*(C_d)$, in the sense that 
$$\{0\}=\theta^{\geq g+1, m}\subseteq \cdots \subseteq \theta^{\geq i+1, m}\subseteq \theta^{\geq i, m}\subseteq \cdots \subseteq \theta^{\geq 0,m}=R^m(C_d)  \hspace{0.3cm} \text{ and } \hspace{0.3cm} \theta^{\geq i, m}\cdot \theta^{\geq j, l}\subseteq \theta^{\geq i+j, m+l}.$$
The properties of the $\theta$-filtration are collected in the following result.

\begin{prop}\label{P:theta}
Let $0\leq m \leq d$ and $0\leq i \leq g+1$. Set as usual $r(m):=\min\{m, d-m, g\}$. Then the following properties hold true.
\begin{enumerate}[(i)]
\item \label{P:theta1}
If $i\leq m+1$ 
then the codimension of $\theta^{\geq i, m}$ inside $R^m(C_d)$ is equal to
$$
\codim \theta^{\geq i, m}=
\begin{cases}
i & \text{ if } r(m)=m \: \text{ or } g, \\
\max\{i-g+d-m,0\} & \text{ if } r(m)=d-m\leq g\leq m,\\
\max\{i-2m+d,0\} & \text{ if } r(m)=d-m\leq m\leq g.
\end{cases}
$$
Moreover, a basis of $\theta^{\geq i, m}$ is given by 
$$\begin{sis}
 \{\theta^ix^{m-i}, \ldots, \theta^mx^0\} \quad & \text{ if } r(m)=m \:\: \text{ and } \: 0\leq i \leq m+1 , \\
 \{\theta^ix^{m-i}, \ldots, \theta^g x^{m-g}\} \quad & \text{ if } r(m)=g \: \:\text{ and } \: 0\leq i \leq g+1, \\
  \{\theta^ix^{m-i}, \ldots, \theta^g x^{m-g}\} \quad & \text{ if } r(m)=d-m\leq g\leq m \:\: \text{ and } \: g-(d-m)\leq i \leq g+1, \\
	\{\theta^ix^{m-i}, \ldots, \theta^mx^0\} \quad & \text{ if } r(m)=d-m\leq m\leq g \: \:\text{ and }  \: 2m-d\leq i \leq m+1. \\
\end{sis}$$
\item \label{P:theta2} 
Under the perfect pairing between $R^m(C_d)$ and $R^{d-m}(C_d)$ given by the intersection product (see Proposition \ref{basetaut}\eqref{basetaut5}), we have that 
$$(\theta^{\geq i, m})^{\perp}  \supseteq\theta^{\geq g+1-i, d-m},$$
with equality if and only if one the following assumptions hold:
\begin{itemize}
\item $g\leq \max\{m,d-m\}$, 
\item $i=g+1$ or $m\leq d-m\leq g$ and $g-(d-m)+m+1\leq i \leq g+1$, in which case the left and right hand side are both equal to $R^{d-m}(C_d)$,
\item $i=0$ or $d-m\leq m\leq g$ and $0 \leq i\leq 2m-d$, in which case the left and right hand side are both equal to zero.
\end{itemize}
\end{enumerate}
\end{prop}
\begin{proof}
Part \eqref{P:theta1} is obvious if either $r(m)=m$ or $r(m)=g$, since in the former case the elements $\{\theta^0x^m, \ldots, 
\theta^m x^0\}$ form a basis of $R^m(C_d)$ while in the latter case the elements $\{\theta^0x^m, \ldots, \theta^gx^{m-g}\}$ form a basis of $R^m(C_d)$ by Proposition \ref{basetaut}\eqref{basetaut4}. On the other hand, if $r(m)=d-m$ then any subset of $(d-m+1)$ elements of $\{\theta^{0}x^{m}, \ldots, \theta^{\min\{g, m\}}x^{m-\min\{g,m\}}\}$ form a basis of $R^m(C_d)$ by Proposition \ref{basetaut}\eqref{basetaut4}. This easily imply \eqref{P:theta1} for $r(m)=d-m$. 

Part \eqref{P:theta2}: the inclusion  
\begin{equation*}
(\theta^{\geq i, m})^{\perp} \supseteq \theta^{\geq g+1-i, d-m}
\end{equation*} 
follows from the relation $\theta^{g+1}=0$. We conclude with a straightforward comparison (left to the reader) of the 
codimensions of $(\theta^{\geq i, m})^{\perp}$ and of $\theta^{\geq g+1-i, d-m}$, using  \eqref{P:theta1}.
\end{proof}



The link between tautological Abel-Jacobi faces and the $\theta$-filtration is clarified in the following  

\begin{prop}\label{P:AJtheta}
Let $0\leq n \leq d$ and  $1+\max\{0,n-g\}\leq r\leq n$. 
\begin{enumerate}[(i)]
\item \label{P:AJtheta1} We have an equality of subcones of $\Psefft_n(C_d)$
\begin{equation}\label{E:AJtheta}
\AJt_n^r(C_d)= (\theta^{\geq n+1-r,n})^\perp \cap \Psefft_n(C_d).
\end{equation}
In particular, $\dim \AJt_n^r(C_d)\leq \dim (\theta^{\geq n+1-r,n})^\perp=\codim (\theta^{\geq n+1-r,n})$.
\item \label{P:AJtheta2} 
If $\dim \AJt_n^r(C_d)=\dim (\theta^{\geq n+1-r,n})^\perp$ then $\AJt_n^r(C_d)$ is a perfect face of $\Psefft_n(C_d)$ whose (perfect) dual face is  $\theta^{\geq n+1-r,n} \cap \Neft^n(C_d)$.
\end{enumerate}
\end{prop}  
When the assumption of \eqref{P:AJtheta2} holds true, the perfect face $\theta^{\geq n+1-r,n} \cap \Neft^n(C_d)$ of $\Neft^n(C_d)$ will be called  \emph{nef $\theta$-face}. 
A nef $\theta$-face of dimension one will be called \emph{nef $\theta$-edge}, and using Proposition 
 \ref{P:theta}\eqref{P:theta1} it is easy to see that a nef $\theta$-edge is equal to 
 $$\theta^{\geq \min\{n,g\},n} \cap \Neft^n(C_d)=\cone(\theta^{\min\{n,g\}}x^{n-\min\{n,g\}}).$$ 
 
\begin{proof}
\eqref{P:AJtheta1}: note that, since $\theta$ is the pull-back via $\alpha_d:C_d\to \Pic^d(C)$ of an ample line bundle on $\Pic^d(C)$, from Definition \ref{contract} it follows that for any $\beta\in \Pseff_n(C_d)$ we have 
 \begin{equation}\label{conttheta} 
 \contr_{\alpha_d}(\beta)\geq r \Leftrightarrow \beta\cdot \theta^{n+1-r}=0. 
 \end{equation}
Therefore, since $\AJt_n^r(C_d)$ is the conic hull of all elements $\beta \in \Psefft_n(C_d)$ having contractibility index at least $r$, formula \eqref{conttheta} implies that  $\AJt_n^r(C_d)\subseteq (\theta^{\geq n+1-r,n})^\perp \cap \Psefft_n(C_d)$.
In order to prove the reverse implication, by contradiction assume that there exists an element $\beta\in \Psefft_n(C_d)$ such that $\beta\in (\theta^{\geq n+1-r,n})^\perp$ and $\beta\cdot \theta^{n+1-r}\neq 0$. The element  $\beta\cdot \theta^{n+1-r}$ lies in $R^{d+1-r}(C_d)$ and, since it is non-zero (which implies that $r\geq 1$), applying Proposition \ref{basetaut}\eqref{basetaut5} we find an element $\gamma\in R^{r-1}(C_d)$ such that $\beta\cdot \theta^{n+1-r} \cdot \gamma\neq 0$. But then, since $ \theta^{n+1-r} \cdot \gamma\in \theta^{\geq n+1-r,n}$, we 
find that $\beta \not\in (\theta^{\geq n+1-r,n})^\perp$, which is the desired contradiction. 


Part \eqref{P:AJtheta2}: if $\dim \AJt_n^r(C_d)=\dim (\theta^{\geq n+1-r,n})^\perp$ then $\langle \AJt_n^r(C_d)\rangle=(\theta^{\geq n+1-r,n})^\perp$, which implies that the dual face of $\AJt_n^r(C_d)$ is equal to 
$$((\theta^{\geq n+1-r,n})^\perp)^\perp \cap \Neft^n(C_d)= \theta^{\geq n+1-r,n} \cap \Neft^n(C_d).$$
 Observe that $\AJt_n^r(C_d)$ is a full cone in $(\theta^{\geq n+1-r,n})^\perp$ by assumption, while $\theta^{\geq n+1-r,n} \cap \Neft^n(C_d)$ is a full cone in $\theta^{\geq n+1-r,n}$ since $\theta$ is nef (hence limit of ample classes) and $x$ is ample. 
Therefore, we can apply Remark \ref{perf} in order to conclude that $\AJt_n^r(C_d)$ and $\theta^{\geq n+1-r,n} \cap \Neft^n(C_d)$
are perfect dual faces.  
\end{proof}

\begin{remark} 
The equality \eqref{E:AJtheta} is true also for the (non-tautological) Abel-Jacobi faces with the same proof (taking orthogonals in $N_n(C_d)$).
\end{remark}


Note that Proposition \ref{P:AJtheta}\eqref{P:AJtheta2} gives a criterion to find perfect faces of $\Psefft_n(C_d)$. Let us see how we could apply this criterion to find facets (which are always perfect) and edges, i.e. one-dimensional perfect faces.

The dimension of  $(\theta^{\geq n+1-r,n})^\perp\subseteq R_n(C_d)$, which 
is equal to the codimension of $\theta^{\geq n+1-r,n}\subseteq R^n(C_d)$, can be computed (in the non trivial range $n+1-r\leq g$) using Proposition \ref{P:theta}\eqref{P:theta1} and it is equal to:
\begin{equation}\label{E:exp-dim}
\dim (\theta^{\geq n+1-r,n})^\perp=\codim \theta^{\geq n+1-r,n}=
\begin{cases}
n+1-r & \text{ if either }  r(n)=n     \text{ or } r(n)=g, \\
\max\{d-g+1-r, 0\} & \text{ if } d-n\leq g\leq n, \\
\max\{d-n+1-r,0\} & \text { if } d-n\leq n\leq g. 
\end{cases}
\end{equation}
Therefore, we find that 
$$\codim (\theta^{\geq n+1-r,n})^\perp=1 \Leftrightarrow \dim (\theta^{\geq n+1-r,n})^\perp=r(n)\Leftrightarrow 
\begin{cases} 
r=1 & \text{ if } n\leq g, \\
r=n+1-g & \text{ if } g\leq n.
\end{cases}
$$
Let us now examine when, in each of the above two cases, we get indeed a tautological Abel-Jacobi facet.  

\begin{prop}\label{P:facet}
\noindent 
\begin{enumerate}[(i)]
\item \label{P:facet1} If $g\leq n$ then $\AJt_n^{n+1-g}(C_d)$ is a facet of $\Psefft_n(C_d)$.
\item \label{P:facet2} If $n\leq g$ then $\AJt^1_n(C_d)$ is a facet of $\Psefft_n(C_d)$ under one of the following assumptions:
\begin{enumerate}[(a)]
\item \label{P:faceta} $\gon_n(C)\leq d$ (which is always satisfied if $g\leq d-n$);
\item \label{P:facetb} $n=g-1$;
\item \label{P:facetc} $g\leq d$ and $C$ is very general over an uncountable base field $k$.
\end{enumerate}
\end{enumerate}
\end{prop}
Note that:  \eqref{P:facet1} (and \eqref{P:faceta} for $g\leq d-n$) is a special case of  Theorem \ref{thetaPseff}, \eqref{P:faceta} is a special case of Theorem \ref{thetaNef}, and 
 \eqref{P:facetb} for $d-n\leq g-1$ (otherwise it belongs to case \eqref{P:faceta}) is a special case of Theorem \ref{T:BNfaces}.

\begin{proof}
As observed above, parts \eqref{P:facet1}, \eqref{P:faceta} and \eqref{P:facetb} are special case of theorems that will be proved later. 

Let us prove part \eqref{P:facetc}. The assumption that $g\leq d$ implies that the Abel-Jacobi morphism $\alpha_d$ is surjective. Hence, using  that $k$ is uncountable (and algebraically closed) and that the fibers of $\alpha_d$ are projective spaces, we can apply \cite[Thm. 1.2]{fl3} in order to conclude that $\langle \AJ^1_n(C_d)\rangle=\ker ((\alpha_d)_*:N_n(C_d)\to N_n(\Pic^d(C)))$. Since $C$ is very general, we have that $N_n(C_d)=R_n(C_d)$ (which also implies that $\AJ^1_n(C_d)=\AJt_n(C_d)$) and $N_n(\Pic^d(C))=\langle [\Theta]^{g-n}\rangle$  (see \cite[Fact 2.6]{BKLV} and Ben Moonen's appendix to \cite{BKLV}).
Therefore, the kernel of $(\alpha_d)_*:N_n(C_d)\to N_n(\Pic^d(C))$ is isomorphic to the linear space of all elements $z\in R_n(C_d)$ such that  $0=(\alpha_d)_*(z)\cdot [\Theta]^n=z\cdot \theta^n=0$,
that is to $(\theta^{\geq n, n})^{\perp}$. Putting everything together, we deduce that $\langle \AJt^1_n(C_d)\rangle=(\theta^{\geq n, n})^{\perp}$, which implies that $ \AJt^1_n(C_d)$ is a facet of $\Psefft_n(C_d)$ (since $(\theta^{\geq n, n})^{\perp}$ has codimension one in $R_n(C_d)$ as observed above). 
\end{proof}
Let us now discuss when Proposition \ref{P:AJtheta}\eqref{P:AJtheta2} can be used to find edges of $\Psefft_n(C_d)$.
Using \eqref{E:exp-dim}, we find that 
$$\dim (\theta^{\geq n+1-r,n})^\perp=1 \Longleftrightarrow 
\begin{cases}
r=n & \text{ if either }  r(n)=n     \text{ or } r(n)=g, \\
r=d-g & \text{ if } d-n\leq g\leq n, \\
r=d-n & \text { if } d-n\leq n\leq g. 
\end{cases}
$$


Let us now check, in each of the above cases, when we can apply the criterion of Proposition \ref{P:nonzero} to conclude that $\AJt_n^ r(C_d)$ is non-zero, and hence that it is an edge of $\Psefft_n(C_d)$.

 We will distinguish the following cases (assuming that $1\leq n\leq d-1$ to avoid trivial faces):
\begin{enumerate}[(A)]
\item If    $g\leq d-n$   then clearly $C_d^n=C_d$ and we deduce that $\AJt_n^n(C_d)$ is non-zero;
\item If $d-n\leq g\leq n$ then clearly $C_d^{d-g}=C_d$ and we deduce that $\AJt_n^{d-g}(C_d)$ is non-zero:
\item If $n\leq d-n<g$ (which implies that $d\leq 2g-2$) then $\AJt_n^n(C_d)$ is non-zero if $C_d^n$ has some tautological irreducible component of maximal dimension  and if 
$n\leq \dim C_d^n=n+\dim W_d^n(C)$, which is equivalent to the non-emptiness of $W_d^n(C)$, or in other words to $d\geq \gon_n(C)$.  
\item If $d-n\leq n<g$ (which implies that $d\leq 2g-2$) then $\AJt_n^{d-n}(C_d)$ is non-zero if $C_d^n$ has some tautological irreducible component of maximal dimension  and if 
$$n\leq \dim C_d^{d-n}=d-n+\dim W_d^{d-n}(C) \Longleftrightarrow \dim W_d^{d-n}(C)\geq 2n-d=d-2(d-n).$$ 
By Martens' theorem  (see \cite[Chap. IV, Thm. 5.1]{GAC1}), this can happen if either $d-n=d-g+1$, i.e. $n=g-1$, 
or $C$ is hyperelliptic.
\end{enumerate}
We will see in the next sections that indeed in all the above cases we get edges of $\Psefft_n(C_d)$: cases (A) and (B) will be analyzed in Section \ref{S:theta} (and indeed Case (A) also follows from Section \ref{S:Neftheta}), case (C)  in Section  \ref{S:Neftheta}, case (D) with $n=g-1$ 
 in Section \ref{S:BN} and case (D) for $C$ hyperelliptic in Section \ref{S:hyper}. 

Quite remarkably, we will see that in all the above cases  the non-trivial tautological Abel-Jacobi faces of $\Psefft_n(C_d)$ form  a \emph{maximal chain of perfect non-trivial faces}, i.e. a chain of perfect non-trivial faces of $\Psefft_n(C_d)$ whose dimensions start from one and increase by one at each step until getting to the dimension of $\Psefft_n(C_d)$ minus one. 

\begin{remark}\label{R:BNperf}
The unique BN rays $\cone([C_d^r])$ to which we can apply Proposition \ref{P:AJtheta}\eqref{P:AJtheta2} are those with $\rho=\rho(g,r,d)=0$ or with $d=g+r-1$.

Indeed, since $\AJt_{r+\rho}^r(C_d)=\cone([C_d^r])$ has dimension one, the hypothesis of Proposition \ref{P:AJtheta}\eqref{P:AJtheta2} does hold true if and only if
$$1=\dim(\theta^{\geq \rho+1, r+\rho})^\perp=\codim \theta^{\geq \rho+1, r+\rho}.$$
Now observe that $\displaystyle d=\frac{\rho+rg}{r+1}+r$ and the hypothesis on $d$ in Theorem \ref{AJtaut} translates into $0\leq \rho\leq g-r-1$. The dimension $n=r+\rho$ and the codimension $m=d-n$ of $C_d^r$ satisfy the following easily checked inequalities
$$
\begin{sis}
& n<g, \\
&m<g,\\
& n\geq m \Longleftrightarrow \frac{r}{2r+1}(g-r-1) \leq  \rho.
\end{sis}
$$
Using this, we can compute the codimension of $\theta^{\geq \rho+1, r+\rho}$ using Proposition \ref{P:theta}\eqref{P:theta1}:
$$\codim \theta^{\geq \rho+1, r+\rho}=
\begin{cases}
\rho+1 & \text{ if } \rho\leq  \frac{r}{2r+1}(g-r-1), \\
d-2r-\rho+1=r(g-d+r-1)+1 & \text{ if } \frac{r}{2r+1}(g-r-1) \leq  \rho.
\end{cases}
$$
Hence we see that $\codim \theta^{\geq \rho+1, r+\rho}=1$ if either $\rho=0$ or $d=g+r-1$.
\end{remark}

\section{The $\theta$-faces}\label{S:theta}

In this section,  we are going to describe the tautological Abel-Jacobi faces of $\Psefft_n(C_d)$ under the assumption that 
$g \le \max\{n,d-n\}$. Note that this assumption is always satisfied if $d>2g-2$ and it is never satisfied if $d<g$. 

Let us start with the following result that gives a lower bound on the dimension of the tautological Abel-Jacobi faces. 

\begin{lemma}\label{L:lowbound}
Let $0\leq n \leq d$ and  $1+\max\{0,n-g\}\leq r\leq n$. 
The cone 
$$\theta^{\geq g-n+r,d-n} \cap \Psefft_n(C_d)\subset \theta^{\geq g-n+r,d-n}\subseteq R_{n}(C_d)$$ 
is contained in $\AJt_n^r(C_d)$ and it is a full-dimensional  cone in  $\theta^{\geq g-n+r,d-n}$.
In particular, we have that 
$$\dim \AJt_n^r(C_d)\geq \dim \theta^{\geq g-n+r,d-n}.$$
 \end{lemma}
\begin{proof}
Since $\theta^{\geq g-n+r,d-n}\subseteq (\theta^{n+1-r,n})^{\perp}$ by Proposition \ref{P:theta}\eqref{P:theta2}, we get that the cone $\theta^{\geq g-n+r,d-n} \cap \Psefft_n(C_d)$ is contained in 
$\AJt^r_n(C_d)$ by \eqref{E:AJtheta}.

By Definition \ref{thetalin}, the linear subspace $\theta^{\geq g-n+r,d-n}\subseteq R_{n}(C_d)$ is generated by monomials in $x$ and $\theta$.  Since $\theta$ is nef (hence limit of ample classes) and $x$ is ample we have that each monomial in $x$ and $\theta$ is a pseudoeffective class. This implies that $\theta^{\geq g-n+r,d-n} \cap \Psefft_n(C_d)$ is a full-dimensional cone in $\theta^{\geq g-n+r,d-n}$. 
\end{proof}

Using the above Lemma, we can now prove the main result of this section.

\begin{thm}\label{thetaPseff}
Let $0 \le n \le d$ and assume that $g \le \max\{n,d-n\}$.
Then the Abel-Jacobi face $\AJt_n^r(C_d)$  is equal to $\theta^{\geq g-n+r, d-n} \cap \Psefft_n(C_d)$, and it is non-trivial if and only if $1 + \max\{0,n-g\}\le r \le \min\{n,d-g\}$, in which case it is a  perfect face of dimension $\min\{n,d-g\}-r+1$ and codimension $r-\max\{n-g,0\}$.
Hence, the following chain
\begin{equation}\label{E:chain-th}
\theta^{\geq {\min\{g,d-n\}}}\cap \Psefft_n(C_d)\subset  \ldots \subset \theta^{\geq g+1-\min\{g,n\}}\cap \Psefft_n(C_d)\subset \Psefft_n(C_d)
\end{equation}
is a maximal chain of perfect non-trivial faces of $\Psefft_n(C_d)$.
The dual chain of \eqref{E:chain-th} is equal to 
\begin{equation}\label{E:nef-th}
\theta^{\geq {\min\{g,n\}}}\cap \Neft^n(C_d)\subset  \ldots \subset \theta^{\geq g+1-\min\{g,d-n\}}\cap \Neft^n(C_d)\subset \Neft^n(C_d).
\end{equation}
\end{thm}

The  faces in \eqref{E:chain-th}  will be called \emph{pseff $\theta$-faces}, while the faces in \eqref{E:nef-th} are the nef $\theta$-faces introduced after Proposition \ref{P:AJtheta}. 
Note that 
$$\cone(\theta^{\min\{g,d-n\}}x^{d-n-\min\{g,d-n\}})=\theta^{\geq {\min\{g,d-n\}}}\cap \Psefft_n(C_d)=\theta^{\geq {\min\{g,d-n\}}}\cap \Neft^{d-n}(C_d)$$ 
is an edge (i.e. perfect extremal ray) of $\Psefft_n(C_d)$, which we will call the \emph{pseff $\theta$-edge}, and it coincides with the nef $\theta$-edge. 
 On the other hand, since the class $x$ is ample, the other monomials in $x$ and $\theta$ cannot generate an extremal ray of either $\Psefft_n(C_d)$ or of $\Neft^{d-n}(C_d)$. 
\begin{proof}
Fix an integer $r$ such that $1+\max\{0,n-g\}\leq r\leq n$.
Using the assumption $g\leq \max\{n,d-n\}$, Proposition \ref{P:theta}\eqref{P:theta2} implies that 
$$(\theta^{\geq n+1-r, n})^\perp=\theta^{\geq g-n+r, d-n}\subseteq R_n(C_d).$$
This, together with Proposition \ref{P:AJtheta}\eqref{P:AJtheta1} and Lemma \ref{L:lowbound}, gives the equality of cones
$$\AJt_n^r(C_d)=\theta^{\geq g-n+r, d-n} \cap \Psefft_n(C_d)$$
and the fact that 
$$\dim \AJt_n^r(C_d)=\dim (\theta^{\geq n+1-r, n})^\perp.$$
Hence we can  apply Proposition \ref{P:AJtheta}\eqref{P:AJtheta2} in order to conclude that $\AJt_n^r(C_d)$ is a perfect face of $\Psefft_n(C_d)$ whose dual face is equal to $\theta^{\geq n+1-r,n} \cap \Neft^n(C_d)$. 

Finally,  Proposition \ref{P:theta}\eqref{P:theta1} gives that the linear subspace $(\theta^{\geq n+1-r, n})^\perp\subseteq R_n(C_d) $ is non-trivial  if and only if $1 + \max\{0,n-g\}\leq r\leq \min\{n,d-g\}$, in which case it has dimension $\min\{n,d-g\}-r+1$. 
\end{proof}

\begin{remark}\label{sharp1} 
Notice that, outside of the range $g \le \max\{n,d-n\}$, the cones $\theta^{\geq i} \cap \Psefft_n(C_d)$ may not be a face of $\Psefft_n(C_d)$. To see this let $m$ be odd and such that $1\leq m \leq g-1$ and let $d = g+m-1$. Now, by  \eqref{E:clBN}, the coefficient of $x^m$ in $[C_{g+m-1}^m]$ is $(-1)^m < 0$ while, for any $m$-codimensional diagonal, the same coefficient is positive by \cite[Prop. 3.1]{BKLV}. Hence, in $\Psefft_{g-1}(C_{g+m-1})$, the class $[C_{g+m-1}^m]$ and the $m$-codimensional diagonals lie in different half-spaces with respect to the hyperplane $\theta^{\geq 1}$, which then implies that $\theta^{\geq 1}\cap \Psefft_{g-1}(C_d)$ is not a face of $\Psefft_{g-1}(C_{g+m-1})$.
\end{remark}

Let us finish this section by giving upper and lower bounds for the dimension of the tautological Abel-Jacobi faces in the numerical ranges not included in the above Theorem \ref{thetaPseff}.

\begin{prop}\label{P:uplow}
Assume that 
$g \ge \max\{n,d-n\}$. Then
\begin{enumerate}[(i)]
\item \label{P:uplow1} $\AJt_n^r(C_d)$ is trivial unless $1\leq r\leq \min\{n, d-n\}$.
\item \label{P:uplow2} If $1\leq r\leq \min\{n, d-n\}$ then 
\begin{equation}\label{E:uplow}
\max\{d+1-g-r,0\}\le \dim \AJt_n^r(C_d)\leq r(n)-r+1.
\end{equation}
In particular, if $1\leq r \leq d-g$ (which forces $g+1\le d$) then $\AJt_n^r(C_d)$ is non-trivial. 
\end{enumerate}
\end{prop}
\begin{proof}
Observe that $\AJt_n^r(C_d)$ is defined only for $1=1+\max\{0,n-g\}\leq r \leq n$. Under this assumption, Proposition \ref{P:AJtheta}\eqref{P:AJtheta1} and Lemma \ref{L:lowbound} give that 
\begin{equation}\label{E:inq1}
\dim \theta^{g-n+r,d-n}\leq \dim \AJt_n^r(C_d)\leq \codim \theta^{\geq n+1-r,n}.
\end{equation}
Using the assumption 
$g \ge \max\{n,d-n\}$ and Proposition \ref{P:theta}\eqref{P:theta1}, we compute 
\begin{equation}\label{E:inq2}
\codim \theta^{\geq n+1-r,n}=
\begin{cases} n+1-r & \text{ if } n \leq d-n \leq g, \cr
\max\{d-n-r+1,0\} & \text{ if } d-n\leq n \leq g. \cr
\end{cases}
\end{equation}
Therefore if $d-n<r$  (which can only happen in the second case) then $\codim \theta^{\geq n+1-r,n}=0$, while if $r\leq d-n$ then 
$\codim \theta^{\geq n+1-r,n}=r(n)-r+1$. Using the upper bound in \eqref{E:inq1}, this implies that 
$\AJt_n^r(C_d)=(0)$ if $d-n<r$ (which proves \eqref{P:uplow1}) and that 
$\dim \AJt_n^r(C_d)\leq r(n)-r+1$ if $r\leq d-n$. 

On the other hand, using again the assumption 
$g \ge \max\{n,d-n\}$ and Proposition \ref{P:theta}\eqref{P:theta1}, we compute 
\begin{equation}\label{E:inq4}
\dim \theta^{\geq g-n+r,d-n}=
\begin{cases}
d+1-g-r & \text{ if } r\leq d-g, \\
0 & \text{ otherwise.  }
\end{cases}
\end{equation}
If we plug this formula into the lower bound in \eqref{E:inq1}, we get the lower bound of part \eqref{P:uplow2}, and this finishes the proof.
\end{proof}
\begin{remark}\label{R:uplow}
Note that the upper bound and lower bound in the above Proposition \ref{P:uplow} (which are always different except in the 
special cases $n=g$ or $d-n=g$, which we exclude in the discussion that follows) can be strict. 
For example:
\begin{itemize}
\item If $d<\gon(C)$ (which implies that $d\leq \frac{g+1}{2}$ by Lemma \ref{L:gonindex}) then Remark \ref{nontr}  gives that $\AJt_n^r(C_d)=\{0\}$ for any $1\leq r\leq \min\{n, d-n\}$, which shows that the lower bound 
in  \eqref{E:uplow} is (trivially) achieved but not the upper bound. 
\item The BN rays of Theorem \ref{AJtaut} do not achieve the lower bound in  \eqref{E:uplow}, which is zero since $d-g+1\leq r$, while they achieve the upper bound only if $\rho(g,r,d)=0$ or $d=g+r-1$ (see Remark \ref{R:BNperf}).
\item In each of the cases specified in Proposition \ref{P:facet}\eqref{P:facet2},  $\AJt_n^1(C_d)$ is a facet, hence its dimension achieves the upper bound in \eqref{E:uplow} but not the lower  bound.
\item We will show in the sequel that the upper bound in \eqref{E:uplow} is  achieved for any $1\leq r\leq \min\{n, d-n\}$ if either $\gon_n(C)\leq d\leq n+g$ (see Theorem \ref{thetaNef}), or if $n=g-1$ and $g\leq d \leq 2g-2$ (see Theorem \ref{T:BNfaces}), or if 
$g > \max\{n,d-n\}$ and $C$ is hyperelliptic (see Theorem \ref{T:hyper}); and in each of these cases, the lower bound is not achieved.   
\end{itemize}
\end{remark}

\section{Subordinate faces}\label{S:Neftheta}

In this section, we are going to describe some of the Abel-Jacobi faces  using subordinate varieties.
 
Recall  that the \emph{subordinate} variety of a linear system $\l$ is defined (set theoretically) as
\begin{equation}\label{D:sublocus}
\Gamma_d(\l):=\{D\in C_d: \:  D\leq E \text{ for some }E\in \l\}.
\end{equation}
There is a natural scheme structure on $\Gamma_d(\l)$ (indeed $\Gamma_d(\l)$ is a determinantal variety) and the class of $\Gamma_d(\l)$ is computed as follows (see 
\cite[Chap. VIII, \S 3]{GAC1}, \cite[\S 1]{kl2} - the proof works over any algebraically closed field).

\begin{fact}\label{subcyc}
Let $\l$ be a $\g_l^s$ on $C$ and fix an integer $d$ such that $l\geq d\geq s$.  Then $\Gamma_d(\l)$ is of pure dimension $s$ and it has class equal to
$$[\Gamma_d(\l)]=\sum_{k=0}^{d-s}\binom{l-g-s}{k} \frac{x^k\theta^{d-s-k}}{(d-s-k)!}\in R_s(C_d). $$
\end{fact}

Using subordinate varieties, we  construct  subvarieties of $C_d$ that are suitably contracted by the 
Abel-Jacobi map $\alpha_d:C_d\to \Pic^d(C)$.

\begin{prop}\label{cyclecontr}
Let $1\leq n\leq d$ with the property that  $d\geq \gon_n(C)$. Fix a linear system $\l$ of degree $d$ and dimension $n$ on $C$. For any $0 \leq i \leq \min\{n, g\}$, consider the embedding $\psi_i : C_{d-i} \to C_d$ defined by $\psi_i(D)=D+ip_0$, where $p_0$ is a fixed point 
of $C$. Then the subvariety 
$$\Gamma_i:=\psi_i(\Gamma_{d-i}(\l)) \subseteq C_d$$ 
has pure dimension $n$, its class is tautological and  equal to 
\begin{equation}\label{E:classT}
[\Gamma_i]=\sum_{k=0}^{d-i-n}\binom{d-g-n}{k} \frac{x^{k+i}\theta^{d-i-n-k}}{(d-i-n-k)!},
\end{equation}
and its image $\alpha_d(\Gamma_i)$ in $\Pic^d(C)$ is irreducible of dimension $i$. 
\end{prop}
Note that the subvarieties $\Gamma_i$ depend on the choices of the linear system $\l$ and of the base point $p_0$, but their classes $[\Gamma_i]$ are independent of these choices. 
\begin{proof}
Note that  $\min\{n,g\}\leq d-n$ by Remark \ref{R:r-gon}, whence we have that $i \le d-n$.  Fact \ref{subcyc} implies that $\Gamma_{d-i}(\l)$ is pure $n$-dimensional, whence so is $\Gamma_i$. Moreover since the image of $\psi_i$ has class equal to $x^i$ and the 
pull-back map $\psi_i^*$ preserves the classes $x$ and $\theta$, the  class of $\Gamma_i$ is obtained by taking the class  of $\Gamma_{d-i}(\l)$ in $R_n(C_{d-i})$ given by Fact \ref{subcyc}, interpreting it as a class in $R_{n+i}(C_d)$ and then multiplying it for $x^i$; in this way we get the formula \eqref{E:classT}.  

The linear system $\l$ is a sublinear system of a complete linear system $|L|$ for some  $L\in \Pic^d(C)$. Consider the $i$-dimensional irreducible subvariety of $\Pic^d(C)$: 
$$V_i:=\{L(-D+ip_0): \: D\in C_{i}\}.$$
We claim that $\alpha_d(\Gamma_i) = V_i$, which will conclude the proof.  In fact if $\L \in \alpha_d(\Gamma_i)$ then there is $D' \in \Gamma_{d-i}(\l)$ such that $\L \cong \O_C(D'+ip_0)$. But there is also $E \in \l$ such that $E \ge D'$, whence, setting $D = E - D'$ we see that $D \in C_i$ and $\L \cong L(-D+ip_0) \in V_i$. Vice versa if $\L \in V_i$ then $\L \cong L(-D+ip_0)$ for some $D \in C_i$. Since $\dim \l = n \ge i$ there is $E \in \l$ such that $E \ge D$. Setting $D' = E - D$ we find that $D' \in C_{d-i}$ and $D' \le E$, so that $D' \in \Gamma_{d-i}(\l)$, $D'+ip_0 \in \Gamma_i$ and $\alpha_d(D'+ip_0) = \O_C(D'+ip_0) \cong \L$.
\end{proof} 

The intersection of the classes $[\Gamma_i]$ with the monomials $\theta^j x^{n-j}$ is easily computed via the projection formula as follows.

\begin{lemma}\label{intSigma}
Let $Z$ be any pure $n$-dimensional subvariety of $C_d$ such that $\dim \alpha_d(Z)=i$. Then 
$$[Z]\cdot \theta^j x^{n-j}=
\begin{cases}
0 & \text{ if }i< j,\\
>0 & \text{ if } i \geq j.  
\end{cases}
$$
\end{lemma}  
\begin{proof}
Observe that,  since $[Z]\cdot \theta^j x^{n-j}\in N_0(C_d)\cong \R$, we have that $[Z]\cdot \theta^j x^{n-j}=(\alpha_d)_*([Z]\cdot \theta^j x^{n-j})\in N_0(\Pic^d(C))\cong \R$. In order to compute 
the last quantity, we use the projection formula for the Abel-Jacobi map $\alpha_d$:
$$(\alpha_d)_*([Z]\cdot \theta^j x^{n-j})=(\alpha_d)_*([Z]\cdot x^{n-j})\cdot [\Theta]^j.$$
Since $x$ is an ample class on $C_d$, for each irreducible component $Z_k$ of $Z$, the class $[Z_k]\cdot x^{n-j}$ can be represented by a $j$-dimensional irreducible subvariety $W_k$ contained in $Z$ such that $\dim \alpha_d(W_k)=  \min\{\dim \alpha_d(Z_k),j\}$. Passing to the pushforward, we get 
$$(\alpha_d)_*([Z_k] \cdot x^{n-j}) = (\alpha_d)_*([W_k])=
\begin{cases}
0 & \text {if } \dim \alpha_d(Z_k) < j , \\
\deg ((\alpha_d)_{|W_k}) \cdot [\alpha_d(W_k)] & \text{if } \dim \alpha_d(Z_k) \geq j.
\end{cases}
$$
Since $\dim \alpha_d(Z)=i$ we get that $\dim \alpha_d(Z_k) \le i$ for every $k$ and there is a $k_0$ such that $\dim \alpha_d(Z_{k_0}) = i$. 
We conclude by observing that, in the case $j\leq i$, we have that $[\alpha_d(W_{k_0})]\cdot [\Theta]^j>0$ because $\dim \alpha_d(W_{k_0}) = j$ and $\Theta$ is ample on $\Pic^d(C)$.  
\end{proof}

\begin{cor}\label{C:Sigmaper}
Let $0\leq n\leq d$ such that $d-n\geq \min\{n,g\}$ and let $\{Z_i\}_{i=0}^{\min\{n,g\}}$ be pure $n$-dimensional subvarieties of $C_d$ such that $\dim \alpha_d(Z_i)=i$. Then the classes $\{[Z_0],\ldots,[Z_{\min\{n,g\}}]\}$ are linearly independent in $N_n(C_d)$ and we have that 
$$\langle [Z_0],\ldots,[Z_i] \rangle^{\perp}\cap R^n(C_d)= \theta^{\geq i+1, n}$$
has codimension $i+1$ in $R^n(C_d)$, for every $0\leq i \leq \min\{n,g\}$. 
\end{cor}
\begin{proof}
The space $R^n(C_d)$ is freely generated by $\{\theta^0x^n,\ldots, \theta^{r(n)}x^{n-r(n)}\}$ by Proposition \ref{basetaut}\eqref{basetaut4}, where $r(n)=\min\{n,g\}$ because of the assumption on $d$. 
Now Lemma \ref{intSigma} implies that 
$$\langle [Z_0],\ldots,[Z_i] \rangle^{\perp}\cap R^n(C_d)= \theta^{\geq i+1, n} \quad \text{ for any } \quad 0\leq i\leq r(n).$$
The subspace $\theta^{\geq i+1, n}\subset R^n(C_d)$ has codimension $i+1$ by Proposition \ref{P:theta}\eqref{P:theta1}.  
If we apply this result to $i=r(n)$ we deduce that the classes $\{[Z_0],\ldots,[Z_{\min\{n,g\}}]\}$ are linearly independent in $N_n(C_d)$ and this concludes the proof.  
\end{proof}

Using the subvarieties in Proposition \ref{cyclecontr}, we can now describe tautological Abel-Jacobi faces under suitable numerical assumptions. 

\begin{thm}\label{thetaNef}
Let $0 \le n \le d, 1+\max\{0,n-g\}\leq r \leq n$ and assume that  $d \ge \gon_n(C)$.
For any $0\leq i \leq \min\{n,g\}$, consider the classes $[\Gamma_i]\in R_n(C_d)$ given by \eqref{E:classT} and set $$\Sigma_{i+1}:=\langle [\Gamma_0],\ldots,[\Gamma_{i}]\rangle \subset R_n(C_d).$$ 
Then $\AJt_n^r(C_d)$ is a non-trivial face, is equal to $\Sigma_{n+1-r}\cap \Psefft_n(C_d)$ and it is a perfect face of dimension $n+1-r$. Hence, the following chain 
\begin{equation}\label{E:chain-sub}
\Sigma_1\cap \Psefft_n(C_d) \subset \Sigma_2\cap \Psefft_n(C_d) \subset \ldots \subset \Sigma_{\min\{n,g\}} \cap \Psefft_n(C_d)\subset \Psefft_n(C_d)
\end{equation}
is a maximal chain of perfect non-trivial faces of $\Psefft_n(C_d)$.
The dual chain of the chain in \eqref{E:chain-sub} is equal to
\begin{equation}\label{E:nef-sub}
\theta^{\ge \min\{n,g\}}\cap \Neft^n(C_d)\subset \theta^{\ge \min\{n,g\}-1} \cap \Neft^n(C_d) \subset \ldots \subset \theta^{\ge 1} \cap \Neft^n(C_d) \subset \Neft^n(C_d).
\end{equation}
\end{thm}
The faces  in \eqref{E:chain-sub} will be called \emph{subordinate faces}, while the faces in \eqref{E:nef-sub} are the nef $\theta$-faces introduced after Proposition \ref{P:AJtheta}. 
Note that 
$$\cone([\Gamma_0])=\Sigma_1\cap \Psefft_n(C_d)$$ 
is an edge (i.e. a perfect extremal ray) of $\Psefft_n(C_d)$, which we call the \emph{subordinate edge}.
On the other hand, we do not expect that  the classes $[\Gamma_i]$ with $0< i\leq \min\{n,g\}$ generate an extremal ray of $\Psefft_n(C_d)$. Using the fact that $x$ is ample, we can prove that they are not extremal, unless, possibly, when $g>d-n\ge n$.


\begin{proof}
Consider the pure $n$-dimensional tautological subvarieties $\{\Gamma_0,\ldots,\Gamma_r\}$ of $C_d$ constructed in Proposition \ref{cyclecontr} (indeed the last subvariety $\Gamma_r$ will be of no use in what follows). 
Since $d-n\geq \min\{n,g\}$ (see Remark \ref{R:r-gon}), we can apply Corollary \ref{C:Sigmaper} and we get that $(\theta^{\geq n+1-r,n})^\perp=\Sigma_{n+1-r}$,
which combined with Proposition \ref{P:AJtheta}\eqref{P:AJtheta1}, gives that 
$$\AJt_n^r(C_d)=\Sigma_{n+1-r}\cap \Psefft_n(C_d).$$
Since $[\Gamma_i]$ are effective classes, we get the following inclusions of cones
\begin{equation}\label{E:incl2}
 \cone([\Gamma_0], \ldots, [\Gamma_{n-r}])  \subseteq\Sigma_{n+1-r}\cap \Psefft_n(C_d) \subset \Sigma_{n+1-r}. 
\end{equation}
Since $\{[\Gamma_0], \ldots, [\Gamma_{n-r}]\}$ is a basis of $\Sigma_{n+1-r}$ by Corollary \ref{C:Sigmaper}, we infer from the inclusions \eqref{E:incl2} that   $\Sigma_{n-r}\cap \Psefft_n(C_d)$ is a full dimensional cone in  $\Sigma_{n-r}$,
and hence it has dimension $n+1-r=\dim (\theta^{\geq n+1-r,n})^\perp$. 
We can therefore apply Proposition \ref{P:AJtheta}\eqref{P:AJtheta2} and get that $\AJt_n^r(C_d)$ is a perfect face of dimension $n+1-r$ whose dual face is equal to $\theta^{\geq n+1-r,n} \cap \Neft^n(C_d)$. 
\end{proof}

\begin{remark}\label{sharpNef}
Let us compare Theorem \ref{thetaNef} with Theorem \ref{thetaPseff} for a given $n$. We are going to use that $\gon_n(C)\leq g+n$ with equality if and only if $n\geq g$, a fact that follows easily from Lemma  \ref{L:gonindex}.
\begin{itemize}
\item If $d\geq n+g$ (which forces $d\geq \gon_n(C)$) then the two theorems coincide. 
\item If $d-n<g\leq n$ then Theorem \ref{thetaPseff} applies while Theorem \ref{thetaNef} does not apply since $d<\gon_n(C)=g+n$ (using that $g\leq n$).  
\item If $n<g$ and $\gon_n(C)\leq d<g+n$ then Theorem \ref{thetaNef} applies but Theorem \ref{thetaPseff} does not apply since $\max\{n,d-n\}<g$.  
\item If $n<g$ and $d<\gon_n(C)$ then neither one of the theorems applies.
\end{itemize}
\end{remark}


\section{Brill-Noether faces in dimension $g-1$}\label{S:BN}


The aim of this subsection is to describe  the tautological Abel-Jacobi faces of $C_d$ in dimension $g-1$. 
We will assume throughout this section that $g\leq d$ (to avoid trivialities) and that $d\leq 2g-2$ since in the case $d>2g-2$ we have a complete description of the tautological Abel-Jacobi faces in Theorem \ref{thetaPseff}. 

We start by using the Brill-Noether varieties in Example \ref{BN=subor} in order to construct  subvarieties of $C_d$ of dimension $g-1$ that are suitably contracted by the Abel-Jacobi morphism $\alpha_d:C_d\to \Pic^d(C)$.

\begin{prop}\label{P:varU}
Let $d$ be such that $g\leq d\leq 2g-2$.  For any $0 \leq i \leq d-g$, consider the embedding $\psi_i : C_{d-i} \to C_d$ defined by $\psi_i(D)=D+ip_0$, where $p_0$ is a fixed point of $C$. Then the subvariety 
$$\Upsilon_i:=\psi_i(C_{d-i}^{d-g+1-i}) \subseteq C_d$$ 
is irreducible of dimension  $g-1$, its class is tautological and  equal to 
\begin{equation}\label{E:classU}
[\Upsilon_i]=\sum_{\alpha=0}^{d-g+1-i}(-1)^{\alpha}\frac{x^{\alpha+i}\theta^{d-g+1-\alpha-i}}{(d-g+1-\alpha-i)!},
\end{equation}
and its image $\alpha_d(\Upsilon_i)$ in $\Pic^d(C)$ has dimension $2g-2-d+i$. 
\end{prop}
Note that the subvarieties $\Upsilon_i$ depend on the choice of the base point $p_0$, but their classes $[\Upsilon_i]$ are independent of this choice. 
\begin{proof}
Note that $C_{d-i}^{d-g+1-i}$ is an irreducible subvariety of $C_{d-i}$ of dimension $g-1$ by Example \ref{BN=subor}, whence $\Upsilon_i$ is an irreducible subvariety of $C_d$ of dimension $g-1$.

The class of $\Upsilon_i$ can be computed starting from \eqref{E:clBN} in the same  way as formula \eqref{E:classT} is obtained in   Proposition \ref{cyclecontr}.

Finally, by Fact \ref{BNclass}\eqref{BNclass0}, the dimension of $\alpha_{d-i}(C_{d-i}^{d-g+1-i})\subset \Pic^{d-i}(C)$ is equal to 
$$\dim \alpha_{d-i}(C_{d-i}^{d-g+1-i})=\dim C_{d-i}^{d-g+1-i}-(d-g+1-i)=2g-2-d+i.$$
Since $ \alpha_d\circ \psi_i$ is obtained by composing $\alpha_{d-i}$ with the isomorphism 
$$
\begin{aligned}
 \Pic^{d-i}(C)&\longrightarrow \Pic^d(C) \\
L & \mapsto L(ip_0),
\end{aligned}
$$ 
we conclude that $\dim \alpha_d(\Upsilon_i)=\dim \alpha_{d-i}(C_{d-i}^{d-g+1-i})=2g-2-d+i.$ 
\end{proof} 

Using the subvarieties in Proposition \ref{P:varU}, we can now describe tautological Abel-Jacobi faces  in dimension $g-1$. 

\begin{thm}\label{T:BNfaces}
Let $d$ be such that $g\leq d\leq 2g-2$.
 For any $0 \leq i \leq d-g$,  consider the classes $[\Upsilon_i]\in R_{g-1}(C_d)$ given by \eqref{E:classU} and set 
$$\Omega_{i+1}:=\langle [\Upsilon_0],\ldots,[\Upsilon_{i}]\rangle \subset R_{g-1}(C_d).$$ 
Then  $\AJt_{g-1}^r(C_d)$ is a non-trivial face if and only if $1\leq r \leq d-g+1$, in which case $\AJt_{g-1}^r(C_d)$ is equal to 
$\Omega_{d-g+2-r}\cap \Psefft_{g-1}(C_d)$ and it is a perfect face of dimension $d-g+2-r$. 

Hence, the following chain 
\begin{equation}\label{E:chain-BN}
\Omega_1\cap \Psefft_{g-1}(C_d) \subset \Omega_2\cap \Psefft_{g-1}(C_d) \subset \ldots \subset \Omega_{d-g+1} \cap \Psefft_{g-1}(C_d)\subset \Psefft_{g-1}(C_d)
\end{equation}
is a maximal chain of perfect non-trivial faces of $\Psefft_{g-1}(C_d)$.
The dual chain of the chain in \eqref{E:chain-BN} is equal to
\begin{equation}\label{E:nef-BN}
\theta^{\ge g-1}\cap \Neft^{g-1}(C_d)\subset \theta^{\ge g-2} \cap \Neft^{g-1}(C_d) \subset \ldots \subset \theta^{\ge 2g-1-d} \cap \Neft^{g-1}(C_d) \subset \Neft^{g-1}(C_d).
\end{equation}
\end{thm}
The faces  in \eqref{E:chain-BN} will be called \emph{BN(=Brill-Noether) faces in dimension $g-1$}, while the faces in \eqref{E:nef-BN} are the nef $\theta$-faces introduced after Proposition \ref{P:AJtheta}. 
Note that 
$$\cone([C_d^{d-g+1}])=\Omega_1\cap \Psefft_{g-1}(C_d)$$ 
is an edge (i.e. a perfect extremal ray) of $\Psefft_{g-1}(C_d)$, which we call the \emph{BN edge in dimension $g-1$}.
On the other hand, since the class $x$ is ample, the classes $[\Upsilon_i]$ with $0< i\leq d-g$ cannot generate an extremal ray of $\Psefft_{g-1}(C_d)$.

Note that from Proposition \ref{P:nonzero}\eqref{BNextr2} it follows that $\cone([C_d^{d-g+1}])$ is also an extremal ray of the entire (non-tautological) cone $\Pseff_{g-1}(C_d)$, although we do not know if it is an edge of the entire cone. 


\begin{proof}
Using that  $d-(g-1)\leq g-1$,  Proposition \ref{P:theta}\eqref{P:theta1} gives that 
\begin{equation}\label{E:dim-span}
\dim (\theta^{\geq g-r,g-1})^\perp=\codim \theta^{\geq g-r,g-1}=\max\{d-g+2-r,0\},
\end{equation}
which, together with Proposition \ref{P:AJtheta}\eqref{P:AJtheta1}, implies that $\AJt_{g-1}^r(C_d)$ is trivial unless 
$1\leq r \leq d-g+1$. Therefore, from now until the end of the proof, we fix an index $r$ satisfying the above inequalities. 

Consider the irreducible $(g-1)$-dimensional tautological subvarieties $\{\Upsilon_0,\ldots,\Upsilon_{d-g}\}$ of $C_d$ constructed in Proposition \ref{P:varU}. 
Applying Lemma \ref{intSigma} and using \eqref{E:dim-span}, we get that $\{[\Upsilon_0],\ldots,[\Upsilon_{d-g}]\}$ are linearly independent in $R_{g-1}(C_d)$ and that, for any $1\leq i \leq d-g+1$,
$$\Omega_i^\perp= \theta^{\geq 2g-2-d+i,g-1}.$$
Combining this with Proposition \ref{P:AJtheta}\eqref{P:AJtheta1}, we get that 
$$\AJt_{g-1}^r(C_d)=\Omega_{d-g+2-r}\cap \Psefft_{g-1}(C_d).$$
Since $[\Upsilon_i]$ are effective classes, we get the following inclusions of cones
\begin{equation}\label{E:incl-con2}
 \cone([\Upsilon_0], \ldots, [\Upsilon_{d-g+1-r}])  \subseteq\Omega_{d-g+2-r}\cap \Psefft_{g-1}(C_d) \subset \Omega_{d-g+2-r}. 
\end{equation}
Since $\{[\Upsilon_0], \ldots, [\Upsilon_{d-g+1-r}]\}$ is a basis of $\Omega_{d-g+2-r}$, we infer from the inclusions \eqref{E:incl-con2} that   $\Omega_{d-g+2-r}\cap \Psefft_{g-1}(C_d)$ is a full dimensional cone in  $\Omega_{d-g+2-r}$,
and hence it has dimension $d-g+2-r=\dim (\theta^{\geq g-r,g-1})^\perp$. 
We can therefore apply Proposition \ref{P:AJtheta}\eqref{P:AJtheta2} and get that $\AJt_{g-1}^r(C_d)$ is a perfect face of dimension $d-g+2-r$ whose dual face is equal to $\theta^{\geq g-r,g-1} \cap \Neft^{g-1}(C_d)$. 
\end{proof}

We will now compare BN rays and BN faces in dimension $g-1$ with pseff $\theta$-faces and subordinate faces. 

\begin{remark}\label{R:compaAJ}
Let us compare  Theorems \ref{T:BNfaces} and  \ref{AJtaut} with Theorems \ref{thetaPseff} and \ref{thetaNef}. 
\begin{itemize}
\item BN faces in dimension $g-1$ and BN rays exist in a range where pseff $\theta$-faces do not exist. 

Indeed, if we are in the numerical range of Theorem \ref{T:BNfaces},  then $n=g-1$ and $1\leq d-n\leq g-1$ which implies that $\max\{n,d-n\}=n=g-1<g$. 
On the other hand, if we are under the hypotheses of Theorem \ref{AJtaut}, then $C_d^r$ has dimension  $n:=r+\rho=d+r(d-g-r)$ and codimension $m:=d-n=r(g+r-d)$. Now it easily checked that 
$$\begin{aligned}
n<g & \Leftrightarrow d<g+r-1+\frac{1}{r+1}, \\
m<g & \Leftrightarrow \frac{r-1}{r}g+r<d, \\
\end{aligned}
$$
and both conditions are satisfied because of the assumptions on $d$. This implies that $g>\max\{n,d-n\}$ in any of the two cases, hence pseff $\theta$-faces are not defined.
\item BN faces in dimension $g-1$  and subordinate faces coexist if only if $d=2g-2$ and $n=g-1$, in which case they are equal.

Indeed, if we are in the numerical range of Theorem \ref{T:BNfaces}, then $n=g-1$ and $d\leq 2g-2$.
On the other hand, if we are in the numerical range of Theorem \ref{thetaNef}, then $d\geq \gon_{g-1}(C)=2g-2$ (see Lemma \ref{L:gonindex}); hence we must have $d=2g-2$ and $n=g-1$. In this case, we have that 
$$\Sigma_i\cap \Psefft_{g-1}(C_{2g-2})=\AJt_{g-1}^{g-i}(C_{2g-2})=\Omega_i\cap \Psefft_{g-1}(C_{2g-2}),$$
for any $1\leq i \leq g-1$. Even more is true, namely that since $C_{2g-2-i}^{g-1-i}=\Gamma_{2g-2-i}(|K_C|)$, we have that $\Gamma_i=\Upsilon_i$ for any $0\leq i\leq g-1$. 

\item  BN rays  can coexist with subordinate faces if and only if $\rho:=\rho(g,r,d)=0$, in which case the Brill-Noether ray $\cone([C_d^r])$ is equal to the subordinate edge $\cone([\Gamma_d(\l)])$, where $\l$ is a linear system of degree $d$ and dimension $r$. 

Indeed, suppose that a BN ray $\cone([C_d^r])\subset \Psefft_{r+\rho}(C_d)$ coexists with the subordinate faces of $\Psefft_{r+\rho}(C_d)$. Then it must happen that $d\geq \gon_{r+\rho}(C)$, which using that $C$ is Brill-Noether general, translates into
$$d=\frac{rg+\rho}{r+1}+r\geq \frac{(r+\rho)g}{r+\rho+1}+r+\rho.$$
Now it is easy to see, using that $\rho\geq 0$ because $C$ is a Brill-Noether general curve, that the above inequality is satisfied if and only if $\rho=0$. In this case, we claim that any subordinate variety $\Gamma_0=\Gamma_d(\l)$ where $\l$ is a $\g_d^r$ (as in Proposition \ref{cyclecontr}) is a fiber of $\alpha_d$ and an irreducible component of $C_d^r$,
and $C_d^r$ is numerically equivalent to a positive multiple of $\Gamma_0$. Indeed, since $\rho=0$ and $C$ is a Brill-Noether general curve,  $C_d^{r+1}=\emptyset$, which implies that any linear system $\l$ of dimension $r$ and degree $d$ is a complete linear system $|L|$ associated to  some $L\in W_d^r(C)$, and clearly $\Gamma_d(|L|)=\alpha_d^{-1}(L)$. 
Moreover, $\Gamma_d(|L|)$ has contractibility index with respect to $\alpha_d$ equal to $r$ (since it has dimension $r$ and it is a fiber of $\alpha_d$), hence it is an irreducible component of $C_d^r$ by Fact \ref{BNclass}\eqref{BNclass0}. Conversely, any irreducible component of $C_d^r$ is  of the form $\Gamma_d(|L|)$ for some $L\in W_d^r(C)$. Since the class of $\Gamma_d(|L|)$ does not depend on the chosen $L\in W_d^r(C)$, we conclude that $[C_d^r]$ is a positive multiple of $[\Gamma_0]$. 
\end{itemize}
\end{remark}

\section{Hyperelliptic curves}\label{S:hyper} 

The aim of this Section is to describe the tautological Abel-Jacobi faces in $\Psefft_n(C_d)$ for $C$ an hyperelliptic curve. 
We will assume throughout this section that $d\leq 2g-2$ since in the case $d>2g-2$ we have a complete description of the tautological Abel-Jacobi faces in Theorem \ref{thetaPseff}. 

A crucial role is played by  Brill-Noether varieties for hyperelliptic curves, which we now study.

\begin{prop}\label{P:BNhyper}
Let $C$ be an hyperelliptic curve of genus $g\geq 2$. Fix integers $d$ and $r$ such that $0\leq d\leq 2g-2$ and $\max\{0,d-g+1\}\leq r \leq \frac{d}{2}$. Then $C_d^r$ is irreducible of dimension $d-r$ and its class is a positive multiple of 
\begin{equation}\label{Cdr-hyper}
\sum_{k=0}^{r}\binom{d-r-g}{k}\frac{x^ k\theta^{r-k}}{(r-k)!}.
\end{equation}
\end{prop}
\begin{proof}
We will denote by $\g^1_2$ the hyperelliptic linear series on $C$, by $\O_C(\g_2^1)$ its associated line bundle and by $\iota$ the hyperelliptic involution on $C$.

Let us distinguish two cases, according to whether or not $d\leq g$. 

If $d\leq g$ then any $\g_d^r$ on $C$ is of the form $r\g^1_2+p_1+\ldots+p_{d-2r}$, where $p_1,\ldots,p_{d-2r}$ are points of $C$ 
such that no two of them are conjugate under the hyperelliptic involution \color{black}(see \cite[p.13]{GAC1}). 
Therefore, $C_d^r$ is the image of the finite morphism
$$
\begin{aligned}
C_r\times C_{d-2r} & \longrightarrow C_d \\
(E,D) \ \ \ \ & \mapsto E+\iota(E)+D,
\end{aligned}
$$
from which we deduce that $C_d^r$ is irreducible of dimension $d-r$. 
Moreover, the class $[C_d^r]$ is a positive multiple (depending on its scheme-structure) of $A^{d-2r}(\Gamma_{2r}(r\g^1_2))$, where $A$ is the push operator of \cite[Def. 2.2]{BKLV} and $\Gamma_{2r}(r\g^1_2)$ is the subordinate variety of \eqref{D:sublocus}. 
Combining  Facts \ref{subcyc} and \cite[Fact 2.9(ii)]{BKLV}, one can easily prove by induction on $0\leq i$ that 
$$A^{i}(\Gamma_{2r}(r\g^1_2))= i! \sum_{k=0}^{r}\binom{r-g+i}{k}\frac{x^ k\theta^{r-k}}{(r-k)!},$$ 
which for $i=d-2r$ gives the desired formula.

If $d>g$ then, using the isomorphism $W_d^r(C)\stackrel{\cong}{\longrightarrow} W_{2g-2-d}^{r-d+g-1}(C)$ obtained by sending $L$ into $\omega_C\otimes L^{-1}$ and the fact that $W_{2g-2-d}^{r-d+g-1}(C)$ is irreducible of dimension equal to $2g-2-d-2(r-d+g-1)=d-2r$ by what proved in the previous case for $C_{2g-2-d}^{r-d+g-1}$,  we get that $W_d^r(C)$ is irreducible 
of dimension equal to $d-2r$. Hence $C_d^r$ is irreducible of dimension $d-r$ by Fact \ref{BNclass}\eqref{BNclass0}. 
Moreover, an effective degree-$d$ divisor $D$ on $C$ belongs to $C_d^r$ if and only if $\omega_C(-D)\in  W_{2g-2-d}^{r-d+g-1}(C)$, which by the previous case is equivalent to 
saying that $\omega_C(-D)=\O_C((r-d+g-1)\g_2^1)(E)$ for some $E\in C_{d-2r}$. Using that $\omega_C=\O_C((g-1)\g_2^1)$, 
we conclude that
$$D\in C_d^r\Longleftrightarrow D+E\in (d-r)\g_2^1 \text { for some }E\in C_{d-2r}.$$ 
Therefore, the class of $C_d^r$ is a positive multiple  of the subordinate variety $\Gamma_d((d-r)\g_2^1)$ whose class is given by \eqref{Cdr-hyper} according to Fact  \ref{subcyc}.
\end{proof}

\begin{cor}\label{C:varhyp}
Let $C$ be an hyperelliptic curve of genus $g\geq 2$ and fix integers $n,d$ such that $0 \le d-n\leq  n <g$ (which implies that $0\leq d \leq 2g-2$).
For any $0 \leq i \leq d-n$, consider the embedding $\psi_i : C_{d-i} \to C_d$ defined by $\psi_i(D)=D+ip_0$, where $p_0$ is a fixed point of $C$. Then the subvariety 
$$\Upsilon_i^H:=\psi_i(C_{d-i}^{d-n-i}) \subseteq C_d$$ 
is irreducible of dimension  $n$, its class is tautological and it is  equal, up  to a positive multiple,  to 
\begin{equation}\label{E:classH}
[\Upsilon_i]^H:=\sum_{k=0}^{d-n-i}\binom{n-g}{k}\frac{x^{k+i}\theta^{d-n-i-k}}{(d-n-i-k)!},
\end{equation}
and its image $\alpha_d(\Upsilon_i^H)$ in $\Pic^d(C)$ has dimension $2n-d+i$.
\end{cor}
Note that the subvarieties $\Upsilon_i^H$ depend on the choice of the base point $p_0$, but their classes $[\Upsilon_i^H]$, which coincide with $[\Upsilon_i]^H$ up to positive multiples, are independent of this choice.
\begin{proof}
Note that $C_{d-i}^{d-n-i}$ is an irreducible subvariety of $C_{d-i}$ of dimension $n$ by Proposition \ref{P:BNhyper}, whence $\Upsilon_i^H$ is an irreducible subvariety of $C_d$ of dimension $n$. 

The class of $\Upsilon_i^H$ can be computed, up to a positive multiple, starting from \eqref{Cdr-hyper} in the same  way as formula \eqref{E:classT} is obtained in   Proposition \ref{cyclecontr}.
Finally, the dimension of  $\alpha_d(\Upsilon_i^H)$ can be computed similarly to what was done in Proposition  \ref{P:varU}. 
\end{proof} 

Using the subvarieties constructed in Proposition \ref{cyclecontr} and the ones constructed in Corollary \ref{C:varhyp}, we can now describe tautological Abel-Jacobi faces 
for hyperelliptic curves. 

\begin{thm}\label{T:hyper}
Let $C$ be an hyperelliptic curve of genus $g\geq 2$ and fix integers $n,d$ such that $0 \le n, d-n <g$ (which implies that $0\leq d \leq 2g-2$). 
\begin{enumerate}[(i)]
\item \label{T:hyper1} Assume that $d\geq 2n$. 

For any $0\leq i \leq \min\{n,g\}$, consider the classes $[\Gamma_i]\in R_n(C_d)$ given by \eqref{E:classT} and set $\Sigma_{i+1}:=\langle [\Gamma_0],\ldots,[\Gamma_{i}]\rangle \subset R_n(C_d)$. 
Then, for any $1\leq r\leq n$,  $\AJt_n^r(C_d)$ is a non-trivial face, is equal to $\Sigma_{n+1-r}\cap \Psefft_n(C_d)$ and it is a perfect face of dimension $n+1-r$. Hence, the following chain 
\begin{equation}\label{E:ch-hyp1}
\Sigma_1\cap \Psefft_n(C_d) \subset \Sigma_2\cap \Psefft_n(C_d) \subset \ldots \subset \Sigma_{n} \cap \Psefft_n(C_d)\subset \Psefft_n(C_d)
\end{equation}
is a maximal chain of perfect non-trivial faces of $\Psefft_n(C_d)$.

\item \label{T:hyper2} Assume that $d\leq 2n$.

 For any $0 \leq i \leq d-n$,  consider the classes $[\Upsilon_i]^H\in R_{n}(C_d)$ given by \eqref{E:classH} and set $\Omega_{i+1}^H:=\langle [\Upsilon_0]^H,\ldots,[\Upsilon_{i}]^H\rangle \subset R_{n}(C_d)$. 
Then  $\AJt_{n}^r(C_d)$ is a non-trivial face if and only if $1\leq r \leq d-n$, in which case $\AJt_{n}^r(C_d)$ is equal to 
$\Omega_{d-n+1-r}^H \cap \Psefft_{n}(C_d)$ and it is a perfect face of dimension $d-n+1-r$. 
Hence, the following chain 
\begin{equation}\label{E:ch-hyp2}
\Omega_1^H\cap \Psefft_{n}(C_d) \subset \Omega_2^H\cap \Psefft_{n}(C_d) \subset \ldots \subset \Omega_{d-n}^H \cap \Psefft_{n}(C_d)\subset \Psefft_{n}(C_d)
\end{equation}
is a maximal chain of perfect non-trivial faces of $\Psefft_{n}(C_d)$.

\end{enumerate}
The dual chain of both  the chains  in \eqref{E:ch-hyp1} and \eqref{E:ch-hyp2} is equal to
\begin{equation}\label{E:nef-hyp}
\theta^{\ge n}\cap \Neft^n(C_d)\subset \theta^{\ge n-1} \cap \Neft^n(C_d) \subset \ldots \subset \theta^{\ge \max\{1,2n-d+1\}} \cap \Neft^n(C_d) \subset \Neft^n(C_d).
\end{equation}
\end{thm}

Note that the faces  in \eqref{E:ch-hyp1} are the subordinate faces introduced in Theorem \ref{thetaNef}, while the faces in \eqref{E:nef-hyp} are the nef $\theta$-faces introduced after Proposition \ref{P:AJtheta}. The faces of \eqref{E:ch-hyp2} are new, and they will be called  \emph{hyperelliptic BN(=Brill-Noether) faces}.
Note that 
$$\cone([C_d^{d-n}])=\Omega_1^H\cap \Psefft_n(C_d)$$ 
is an edge (i.e. a perfect extremal ray) of $\Psefft_n(C_d)$, which we call the \emph{hyperelliptic BN(=Brill-Noether) edge}.

Note that from Proposition \ref{P:nonzero}\eqref{BNextr2} it follows that the hyperelliptic BN edge $\cone([C_d^{d-n}])$ is also an extremal ray of the entire (non-tautological) cone $\Pseff_{n}(C_d)$, although we do not know if it is an edge of the entire cone. 

\begin{proof}
Part \eqref{T:hyper1} follows from  Theorem \ref{thetaNef}, using that $\gon_n(C)=2n$ for $C$ hyperelliptic and $n<g$ by Lemma \ref{L:gonindex}. 

Let us now prove part \eqref{T:hyper2}. 
Using that  $d-n\leq n\leq g$,  Proposition \ref{P:theta}\eqref{P:theta1} gives that 
\begin{equation}\label{E:dim-sp}
\dim (\theta^{\geq n+1-r,n})^\perp=\codim \theta^{\geq n+1-r,n}=\max\{d-n+1-r,0\},
\end{equation}
which, together with Proposition \ref{P:AJtheta}\eqref{P:AJtheta1}, implies that $\AJt_{n}^r(C_d)$ is trivial unless 
$1\leq r \leq d-n$. Therefore, from now until the end of the proof, we fix an index $r$ satisfying the above inequalities. 

Consider the irreducible $n$-dimensional tautological subvarieties $\{\Upsilon_0^H,\ldots,\Upsilon_{d-n}^H\}$ of $C_d$ constructed in Corollary \ref{C:varhyp}. 
Applying Lemma \ref{intSigma} and using \eqref{E:dim-sp}, we get that $\{[\Upsilon_0]^H,\ldots,[\Upsilon_{d-n}]^H\}$ are linearly independent in $R_{n}(C_d)$ 
and that, for any $1\leq i \leq d-n$,
$$(\Omega_i^H)^\perp= \theta^{\geq 2n-d+i,n}.$$
Combining this with Proposition \ref{P:AJtheta}\eqref{P:theta1}, we get that 
$$\AJt_{n}^r(C_d)=\Omega_{d-n+1-r}^H\cap \Psefft_{n}(C_d).$$
Since $[\Upsilon_i]^H$ are $\bbQ$-effective classes, we get the following inclusions of cones
\begin{equation}\label{E:in-con}
 \cone([\Upsilon_0]^H, \ldots, [\Upsilon_{d-n-r}]^H)  \subseteq\Omega^H_{d-n+1-r}\cap \Psefft_{n}(C_d) \subset \Omega^H_{d-n+1-r}. 
\end{equation}
Since $\{[\Upsilon_0]^H, \ldots, [\Upsilon_{d-n-r}]^H\}$ is a basis of $\Omega_{d-n+1-r}^H$, we infer from the inclusions \eqref{E:in-con} that   $\Omega^H_{d-n+1-r}\cap \Psefft_{n}(C_d)$ is a full dimensional cone in  $\Omega^H_{d-n+1-r}$, and hence it has dimension $d-n+1-r=\dim (\theta^{\geq n+1-r,n})^\perp$. 
We can therefore apply Proposition \ref{P:AJtheta}\eqref{P:AJtheta2} and get that $\AJt_{n}^r(C_d)$ is a perfect face of dimension $d-n+1-r$ whose dual face is equal to $\theta^{\geq n+1-r,n} \cap \Neft^{n}(C_d)$. 
\end{proof}

\section*{Acknowledgements}
We would like to thank Dawei Chen, Izzet Coskun and Brian Lehmann for helpful discussions.


\begin{thebibliography}{EGKH02}

\bibitem[ACGH]{GAC1} A.  Arbarello, M.  Cornalba, P.A. Griffiths, J. Harris: {\it Geometry of algebraic curves.} Vol. I. Grundlehren der Mathematischen Wissenschaften [Fundamental Principles of Mathematical Sciences], 267. Springer-Verlag, New York, 1985. 



\bibitem[BKLV17]{BKLV} F.  Bastianelli, A.F.  Lopez, A. Kouvidakis, F. Viviani: {\it Effective cycles on the symmetric product of a curve, I: the diagonal cone.} Preprint 2017. 

\bibitem[BFJ09]{BFJ} S. Boucksom, C. Favre, M. Jonsson: {\it Differentiability of volumes of divisors and a problem of Teissier.} J. Algebraic Geom. 18 (2009), no. 2, 279-308.

\bibitem[CC15]{CC} D. Chen, I.  Coskun: \emph{Extremal higher codimension cycles on moduli spaces of curves.}
Proc. Lond. Math. Soc. (3) 111 (2015), no. 1, 181--204. 

\bibitem[DELV11]{DELV} O. Debarre, L. Ein, R. Lazarsfeld, C. Voisin: \emph{Pseudoeffective and nef classes on abelian varieties.}
Compositio Math. 147 (2011), 1793-1818.

\bibitem[DJV13]{DJV} O. Debarre, Z. Jiang, C. Voisin: \emph{Pseudo-effective classes and pushforwards.}
Pure Appl. Math. Q. 9 (2013), no. 4, 643-664.

\bibitem[Ful11]{Ful} M. Fulger: \emph{The cones of effective cycles on projective bundles over curves.} Math. Z. 269 (2011), no. 1-2, 449--459.
 
\bibitem[FL16]{fl2} M.~Fulger, B.~Lehmann. 
\textit{Morphisms and faces of pseudo-effective cones}. 
Proc. Lond. Math. Soc. 112 (2016), no. 4, 651--676.

\bibitem[FL17a]{fl1} M.~Fulger, B.~Lehmann. 
\textit{Positive cones of dual cycle classes}. 
Algebr. Geom. 4 (2017), no. 1, 1--28.

\bibitem[FL17b]{fl3} M.~Fulger, B.~Lehmann. 
\textit{Kernels of numerical pushforwards}. 
Adv. Geom. 17 (2017), no. 3, 373--378. 

\bibitem[FL81]{fl} W.~ Fulton, R.~Lazarsfeld.
\textit{On the connectedness of degeneracy loci and special divisors}. 
Acta Math. 146 (1981), no. 3-4, 271-283.

\bibitem[Gie82]{gies} D. Gieseker: \emph{Stable curves and special divisors: Petri's conjecture.} Invent. Math. 66 (1982), no. 2, 251--275.
 
\bibitem[KL72]{kl1} S.L. Kleiman, D. Laksov: \emph{On the existence of special divisors.} Amer. J. Math. 94 (1972), 431--436. 

\bibitem[KL74]{kl2} S.L. Kleiman, D. Laksov: \emph{Another proof of the existence of special divisors.} Acta Math. 132 (1974), 163--176. 

\bibitem[Kou93]{Kou} A. Kouvidakis: \emph{Divisors on symmetric products of curves.} Trans. Amer. Math. Soc. 337 (1993), no. 1, 117--128. 


\bibitem[Mar67]{mar} H.H. Martens: \emph{On the varieties of special divisors on a curve.} J. Reine Angew. Math. 227 1967 111--120.

\bibitem[Mus11a]{Mus1} Y. Mustopa: \emph{Residuation of linear series and the effective cone of $C_d$.} Amer. J. Math. 133 (2011), no. 2, 393--416.
 
\bibitem[Mus11b]{Mus2} Y.  Mustopa: \emph{Kernel bundles, syzygies of points, and the effective cone of $C_{g-2}$.} Int. Math. Res. Not. IMRN 2011, no. 6, 1417--1437.


\bibitem[Oss14]{oss} B. Osserman: \emph{A simple characteristic-free proof of the Brill-Noether theorem.} Bull. Braz. Math. Soc. (N.S.) 45 (2014), no. 4, 807--818.

\bibitem[Ott12]{Ott1} J.C. Ottem: \emph{Ample subvarieties and $q$-ample divisors.} Adv. Math. 229 (2012), no. 5, 2868-2887.

\bibitem[Ott16]{Ott2} J.C. Ottem: \emph{On subvarieties with ample normal bundle.} J. Eur. Math. Soc. (JEMS) 18 (2016), no. 11, 2459-2468.

\bibitem[Ott15]{Ott3} J.C. Ottem: \emph{Nef cycles on some hyperkahler fourfolds.} Preprint arXiv:1505.01477.
  
\bibitem[Pac03]{Pac} G. Pacienza: \emph{On the nef cone of symmetric products of a generic curve.} Amer. J. Math. 125 (2003), no. 5, 1117--1135.

\bibitem[Pet09]{Pet} T. Peternell: \emph{Submanifolds with ample normal bundles and a conjecture of Hartshorne.} In Interactions of classical and numerical algebraic geometry, Contemporary Mathematics, vol. 496 (American Mathematical Society, Providence, RI, 2009), 317-330.

\bibitem[Voi10]{Voi} C. Voisin: \emph{Coniveau 2 complete intersections and effective coness.} Geom. Funct. Anal. 19 (2010), 1494-1513.

\end{thebibliography}
\end{document}